\documentclass[12pt,a4paper]{article}
\usepackage[utf8]{inputenc}
\usepackage{amsmath}
\usepackage{amsfonts}
\usepackage{amssymb}
\usepackage{amsthm}
\usepackage[colorlinks=true,linkcolor=blue]{hyperref}
\hypersetup{
	citecolor={blue}
}
\usepackage[left=2cm,right=2cm,top=2cm,bottom=2cm]{geometry}
\usepackage{xcolor}
\usepackage[all,cmtip]{xy}
\usepackage{tikz}
\usepackage{tikz-cd}
\usetikzlibrary{decorations.markings}
\usepackage{enumerate}
\usepackage{enumitem}
\setlist[enumerate]{label = (\roman*)}

\newcommand{\U}{\mathrm{U}}

\newcommand{\SU}{\mathrm{SU}}

\newcommand{\SO}{\mathrm{SO}}

\newcommand{\RR}{\mathbb{R}}
\newcommand{\CC}{\mathbb{C}}
\newcommand{\ZZ}{\mathbb{Z}}

\newcommand{\HH}{\mathbb{H}}

\newtheorem{thm}{Theorem}[section]
\newtheorem{prop}[thm]{Proposition}
\newtheorem{cor}[thm]{Corollary}
\newtheorem{lem}[thm]{Lemma}

\theoremstyle{definition}
\newtheorem{rem}[thm]{Remark}
\newtheorem{defn}[thm]{Definition}

\begin{document}

\title{On torus equivariant $S^4$-bundles over $S^4$\\ and Petrie-type questions for GKM manifolds}
\author{Oliver Goertsches\footnote{Philipps-Universit\"at Marburg, email:
goertsch@mathematik.uni-marburg.de}, Panagiotis Konstantis\footnote{Philipps-Universit\"at Marburg, 
email: pako@mathematik.uni-marburg.de}, and Leopold
Zoller \footnote{Universitat de Barcelona, email: leopold.zoller@ub.edu}}

\maketitle

\begin{abstract} 
We classify $T^2$-GKM fibrations in which both fiber and base are the GKM graph of $S^4$, with standard weights in the base. For each case in which the total space is orientable, we construct, by explicit clutching, a realization as a $T^2$-equivariant linear $S^4$-bundle over $S^4$. We determine which of the total spaces of these examples are non-equivariantly homotopy equivalent, homeomorphic or diffeomorphic, thereby finding many examples of a) pairs of homotopy equivalent, non-homeomorphic GKM manifolds with different first Pontryagin class, and b) pairs of GKM actions on the same smooth manifold whose GKM graphs do not agree as unlabeled graphs.
\end{abstract}
\section{Introduction}

The Petrie conjecture \cite{Petrie} states that if a manifold $M$ is homotopy equivalent to $\CC P^n$ and admits a nontrivial circle action, then any homotopy equivalence between $M$ and $\CC P^n$ preserves their Pontryagin classes. Variants of this conjecture were proposed for instance for toric manifolds, see \cite[Problem 5]{MasudaSuh}. A related problem of the same flavour is the cohomological rigidity problem for toric manifolds: it asks whether two toric manifolds with isomorphic (non-equivariant) cohomology rings are necessarily homeomorphic \cite[Problem 1]{MasudaSuh}.

We wish to study the situation through the lens of GKM theory, which is a natural generalization of the toric setting.
Let $T$ be a compact torus. A GKM manifold is a closed orientable manifold with a certain type of $T$-action to which one can associate a labelled graph, the so-called GKM graph. While the GKM definition is generally a lot more flexible than that of a toric manifold, there are strong parallels, with the GKM graph corresponding to the $1$-skeleton of the moment polytope. E.g., under suitable conditions, the GKM graph encodes the (equivariant) cohomology ring \cite{MR375357, MR2308029} as well as the (equivariant) characteristic classes, see \cite{goertsches2022lowdimensional, SWclasses}.

The motivation of the present paper is to provide families of examples of GKM manifolds which display interesting behaviour with regards to their (non-equivariant) homotopy and homeomorphism types, thus providing negative answers to the GKM versions of the above Pertie-type questions. This is achieved by studying the situation of $T$-equivariant $S^4$-bundles over the standard $T^2$-action on $S^4$, which we find to be of independent GKM theoretic interest.

As a first step  we study possible GKM graphs of these bundles from a combinatorial perspective. It was shown in \cite[Proposition 3.7]{MR4634089} that smooth $T$-equivariant fiber bundles $\pi:M\to B$ in which both $M$ and $B$ are GKM $T$-manifolds, give rise to a GKM fibration \cite{GuilleminSabatiniZara} on the level of GKM graphs. Restricting to the case of $T$-equivariant linear $S^4$-bundles $M^8 \to S^4$ we obtain GKM fibrations over biangles with fiber a biangle. We classify in Section \ref{sec:combinatorial} all GKM fibrations in which both fiber and base are biangles, with standard labels in the base. Despite the simplicity of fiber and base, these GKM fibrations fall into $10$ different families, distinguished by a) whether the underlying unlabelled graph is a product or not, b) the behaviour of the connection along horizontal edges, and c) the number of occurring signs in the GKM congruence relations. An additional obstruction to realizability is that the GKM graph of the total space needs to be orientable in the sense of \cite[Section 2.3]{GKMcorrespondence}, see Corollary 2.24 therein. Orientability of the GKM graph of the total space holds true for exactly $5$ of the above $10$ families.
We prove in Section \ref{sec:construction} that these necessary combinatorial conditions are in fact sufficient:
\begin{thm}\label{thm: A}
Every orientable GKM graph that is a GKM fibration with biangle fibers over the GKM graph of the standard $T^2$-action on $S^4$ is realizable as an equivariant linear $S^4$-bundle over $S^4$.
\end{thm}
The proof proceeds by giving an explicit clutching construction for each of the 5 remaining families. This extends known realization results: \cite[Theorem 5.1]{MR4634089} and \cite[Theorem 1.1]{GKMcorrespondence}, which are tailored to dimension 6, and \cite[Theorem 1.1]{GKMflagbundles}, which focuses on GKM fibrations whose fiber is the graph of a generalized flag manifold. It is worth noting that in the flag case, the previous reference shows that the analogue of Theorem \ref{thm: A} does not hold: Beyond dimension 6, not every orientable GKM fibration is realizable by a $T$-equivariant fiber bundle of GKM manifolds. In general it is not known whether any GKM graph which satisfies the known combinatorial obstructions is indeed realizable by a GKM manifold. In a previous paper  \cite{GKMnonrigid} we constructed GKM actions on other $8$-dimensional total spaces of sphere bundles over spheres, namely the two $S^2$-bundles over $S^6$, in order to find examples of GKM actions with exotic behaviour. Namely, those examples provided GKM actions on manifolds with identical GKM graph which are not homotopy equivalent. The examples we consider in the paper at hand display very different behaviour -- in particular, it will turn out that they are determined up to diffeomorphism by their GKM graph.
 
Using GKM theory we determine in Section \ref{sec:topclassification} the first Pontryagin class of all occurring total spaces, which, as described in Section \ref{SubSec: Sphere bundles}, yields the clutching class in $\pi_3(\SO(5))$ (up to sign) -- see Theorem \ref{thm: clutching numbers} for the explicit values.
Furthermore a result of James--Whitehead \cite[p.\ 217]{JamesWhitehead} answers the question which of the corresponding spaces are homotopy equivalent in terms of the clutching class. With regards to our original motivation of finding examples we arrive at the following theorem, proven at the end of Section \ref{sec:topclassification}, which gives a negative answer to possible Petrie-type questions on GKM manifolds:
\begin{thm}\label{thm:examples}
Among the GKM actions on the total space of equivariant linear $S^4$-bundles over the standard action on $S^4$ there are examples of GKM manifolds which are (non-equivariantly) homotopy equivalent but no homotopy equivalence preserves the first Pontryagin class. In particular they are not homeomorphic.
\end{thm} 


While thus far we considered actions of $T^2$ (the minimal dimension for a GKM action, in the final Section \ref{sec:extensions} we complete the picture with regards to higher dimensional tori. More specifically in Theorem \ref{thm: extensions}, we determine for each of the $5$ families in our classification of $T^2$-equivariant bundles the maximal $k$ such that the actions can be extended to an effective $T^k$-action and furthermore prove that this extension can be realized in our examples. Since every GKM $T^k$-action restricts to a GKM $T^2$-action, this settles the previously considered classification and realization questions for all torus dimensions.
\\

 \noindent {\bf Acknowledgements:} We gratefully acknowledge funding of the Deutsche Forschungsgemeinschaft (DFG, German Research Foundation) -- 452427095. We wish to thank Shintaro Kuroki for pointing out to us the toric examples mentioned in Remark \ref{rem:blowup}. Furthermore, we are extremely grateful to G.\ Back for the pleasant working atmosphere.

\section{Preliminaries}

\subsection{GKM actions}\label{SubSec: GKM actions}

The notion of a GKM action dates back to the seminal paper \cite{MR1489894} of Goresky, Kottwitz
and MacPherson. These are certain torus actions on manifolds; the
motivation for their definition is that their equivariant cohomology is determined by the
action on their equivariant one-skeleton, i.e., the union of orbits of dimension at most $1$, which in turn can be encoded in purely combinatorial data.
In \cite{MR1823050}, Guillemin and Zara synthesized this data into an independent combinatorial object
called an \emph{(abstract) GKM graph}. A main theme in GKM theory is to understand the interplay
between GKM graphs and properties of GKM manifolds. Let us start with introducing the
basic notions of GKM theory.

\begin{defn}\label{D: GKM}
Let $M$ be a $2n$-dimensional compact and orientable manifold.
Furthermore, let $T=T^{k}$ be a torus of rank $k$ acting on $M$. If
\begin{enumerate}[label=(\alph*)]
	\item $H^{\mathrm{odd}}(M;\mathbb{Z}) = 0$,
	\item  the fixed point set $M^{T} = \{ p \in M \mid T \cdot p = \{p\}\}$ is finite and if
	\item the \emph{equivariant $1$-skeleton} $M_{1} = \{ p \in M\mid  \dim(T \cdot p) \leq 1\}$ is a finite 
		union of $T$-invariant $2$-spheres,
\end{enumerate}
then we call $(M,T)$ a \emph{GKM manifold} and the action is called a \emph{GKM action}.
\end{defn}

Consider a GKM action of a torus $T$ on a manifold $M$, and $p \in M^{T}$. Choose a complex structure on $T_{p}M$ invariant under the isotropy representation. As a complex representation the action of $T$ on $T_{p}M$ decomposes as
\[
	T_{p}M = \bigoplus_{\lambda} V_{\lambda}
\]
into $n$ summands, where $\lambda \in \operatorname{Hom}(T,S^{1}) \cong \mathbb{Z}^{k}$ are
the \emph{characters} and 
\[
	V_{\lambda} = \{ v \in V \mid t \cdot v = \lambda(t)v \textrm{ for all } t \in T\} .
\]
Sometimes we will not distinguish between the characters and the \emph{weights} which are given as
the differentials of the characters. Indeed, there is a canonical isomorphism
\[
	\operatorname{Hom}(T,S^{1}) \to \operatorname{Hom}(\mathbb{Z}_{\mathfrak t},
	\mathbb{Z}) = \mathbb{Z}^{\ast}_{\mathfrak t}, \quad \lambda \mapsto d
	\lambda|_{\mathbb{Z}_{\mathfrak t}} =: \alpha,
\]
where $\mathfrak t$ is the Lie algebra of $T$ and $\mathbb{Z}_{\mathfrak t} = \ker
\mathrm{exp}$ the \emph{weight lattice}.
The weights as a real representation are given by $\pm \alpha$
and the respective weight spaces are $V_{\pm \alpha} = V_{\alpha} \oplus V_{-\alpha}$. 
Consequently, we consider the weights as elements in $\mathbb{Z}_{\mathfrak t}^{\ast}/\pm
1$. 

The condition that the equivariant one-skeleton is a finite union of $2$-spheres implies that every $V_{\alpha}$ corresponds to an invariant $2$-sphere in $M_{1}$, as it is
the tangent space of such a sphere in a fixed point. The weight $\alpha$
determines the action of $T$ on its respective invariant sphere. It also follows that any two
weights at a fixed point are linearly independent.

\begin{rem}\label{R: ineffective torus action}
In Definition \ref{D: GKM} we did not require that the torus acts effectively. In case the
GKM action is not effective, we may consider the kernel $K\subset T$ of the action, i.e., the closed subgroup of elements that act trivially on $M$. Its Lie algebra $\mathfrak k$ is the Lie subalgebra of $\mathfrak t =
\mathrm{Lie}(T)$ given by the common kernel of all occuring weights. The action of $T/K$ is again of GKM type, and effective.
\end{rem}

The quotient $M_{1}/T$ is homeomorphic to a graph whose vertex set is the fixed point
set $M^{T}$; two vertices are connected by an edge for every $T$-invariant $2$-sphere that connects them. It turns out that this graph together with the respective
weights along the invariant $2$-spheres is a key combinatorial object which encodes
 properties of GKM manifolds. We now recall the abstract definition of GKM graph and explain afterwards how a
GKM manifold induces such a graph. 

\begin{defn}  \label{D: GKM graph}
Let $\Gamma$ be an $n$-valent graph without loops and denote by $V(\Gamma)$ the set of vertices of
$\Gamma$ and by $E(\Gamma)$ the set of edges. We include each edge twice, once for each
orientation. For $e \in E(\Gamma)$, let $\overline{e}$ denote the edge $e$ with
opposite orientation. Furthermore, let $i(e)$ denote the initial vertex of $e$ and $t(e)$
its terminal vertex and for $v \in V(\Gamma)$, let $E_v$ be the set of edges with initial vertex $v$. A \emph{connection} $\nabla$ on $\Gamma$ is a bijective map $\nabla_{e} \colon E_{i(e)}
\to E_{t(e)}$ such that for all $e \in E(\Gamma)$
\begin{enumerate}[label=(\roman*)]
	\item $\nabla_{e}e = \overline{e}$,
	\item $(\nabla_{e})^{-1} = \nabla_{\overline{e}}$
\end{enumerate}
holds. 

A \emph{GKM graph} is a pair $(\Gamma, \alpha)$, consisting of a graph $\Gamma$ as above and an \emph{axial function}
$\alpha \colon E(\Gamma) \to \mathbb{Z}^{k}/\pm 1$ such that the following holds:
\begin{enumerate}[label=(\alph*)]
	\item For every $v \in V(\Gamma)$ and every $e,f \in E(\Gamma)_{v} = \{ e \in E(\Gamma) \mid
		i(e)=v\}$ with $e \neq f$ the elements $\alpha(e)$ and $\alpha(f)$ are linearly
		independent (note that linear independence in this setting is well-defined via the map
		$\mathbb{Z}^{k} \to \mathbb{Z}^{k}/ \pm 1$).
	\item For all $e \in E(\Gamma)$ we have $\alpha(\overline{e}) = \alpha(e)$.
		\item  There is a connection $\nabla$ on $\Gamma$ \emph{compatible} with $\alpha$. That is, if $v \in
		V(\Gamma)$, $e,f \in E(\Gamma)_{v}$ and if for any lift $\tilde \alpha \colon
		E(\Gamma) \to \mathbb{Z}^{k}$ of $\alpha$ along $\mathbb{Z}^{k} \to \mathbb{Z}^{k}/\pm
		1$ there are $\varepsilon \in \{ \pm 1\}$ and $c \in \mathbb{Z}$ such that
		\[
			\tilde \alpha(\nabla_{e}f) = \varepsilon \tilde \alpha(f) + c \tilde \alpha(e)
		\]
		holds.
\end{enumerate}
Finally, we call a GKM graph $(\Gamma, \alpha)$ \emph{effective} if at one vertex $v$ (and hence for
all) the values of the axial function at $E(\Gamma)_v$ generate $\ZZ^k$. 
\end{defn}

Suppose $M$ is a GKM manifold, where $T$ is a torus of rank $k$ acting on $M$. As described
above, $M_{1}/T$ is homeomorphic to a graph $\Gamma$. The weights of the isotropy
representations at every fixed point determine an axial function $\alpha$ on $\Gamma$. More precisely,
the weights are elements of $\operatorname{Hom}(\mathbb{Z}_{\mathfrak t}, \mathbb{Z}) =
\mathbb{Z}^{\ast}_{\mathfrak t}$, defined up to sign. Choosing an isomorphism $T \cong S^{1} \times \ldots
\times S^{1}$ we may identify $\mathbb{Z}^{\ast}_{\mathfrak t}$ with $\mathbb{Z}^{k}$,
which then gives a map $\alpha \colon E(\Gamma) \to \mathbb{Z}^{k}/ \pm 1$.
Furthermore, in  \cite[Proposition 2.3]{MR4363804} and in \cite{MR1823050} it was proven
that there always exists a compatible connection. Hence, we may associate to every GKM
manifold a GKM graph. Effectivity of the GKM graph corresponds to the condition of the $T$-action being effective, cf.\ Remark \ref{R: ineffective torus action}.

\begin{rem}\label{rem:signedunsigned}
\begin{enumerate}
\item
The choice of connection is sometimes defined as part of the data of an abstract GKM graph. We choose not to fix it in the definition since it is in general not canonical (see \ref{rem: overlap} for a discussion of the ambiguity in our examples).
\item
In the literature on GKM theory, often torus actions on almost complex or symplectic manifolds are considered. In these situations, the weights of the isotropy representation become elements of $\mathbb{Z}^{\ast}_{\mathfrak t}$, resulting in a slightly different definition of a GKM graph in which one does not need to consider any lift of the axial function, and in which the signs $\varepsilon$ in Definition \ref{D: GKM graph} are always identically $1$. See, e.g., \cite{MR1823050}. In order to distinguish these situations, one also speaks about \emph{signed} and \emph{unsigned} GKM graphs. In this article, we only consider unsigned GKM graphs; indeed, $S^4$ does not admit any almost complex structure.
\end{enumerate}
\end{rem}

As observed in \cite[Section 2.3]{GKMcorrespondence} the orientability condition of a GKM manifold puts restrictions on the GKM graph of a GKM manifold. Given a  GKM graph $(\Gamma,\alpha)$, let us choose a compatible connection $\nabla$ as well as an arbitrary lift of the axial function $\alpha:E(\Gamma)\to \ZZ^k/\pm 1$ to a function $\tilde{\alpha}:E(\Gamma)\to \ZZ^k$ with $\tilde{\alpha}(\overline{e})=\tilde{\alpha}(e)$ for all $e$. We construct a map $\eta:E(\Gamma)\to \{\pm 1\}$ in the following way: for $e\in E(\Gamma)$, denote by $e_1:=e,e_2,\ldots,e_n$ the edges of $\Gamma$ emerging from $i(e)$. Then there are unique $\epsilon_i\in\{\pm 1\}$, $i=2,\ldots,n$, satisfying 
\[\tilde{\alpha}(\nabla_e(e_i))\equiv \varepsilon_i\tilde{\alpha}(e_i)\mod \tilde{\alpha}(e)\]
in $\mathbb{Z}^k$. We put $\eta(e)=-\varepsilon_2\cdot\ldots\cdot\varepsilon_n$.
\begin{defn} \label{defn:orientable}
	We call $(\Gamma,\alpha)$ \emph{orientable} if there is a choice of lift $\tilde{\alpha}$ of $\alpha$ and compatible connection $\nabla$ such that for every closed edge path $e_1,\ldots,e_m$ in $\Gamma$ one has \[\prod_{i=1}^m \eta(e_i)=1.\]
\end{defn}
It was observed in \cite{GKMcorrespondence} that for an orientable GKM graph any choice of lift $\tilde{\alpha}$ as above and compatible connection $\nabla$ will fulfil the condition in the definition. Then Corollary 2.24 in \cite{GKMcorrespondence} states:
\begin{prop}\label{prop:GKMgraphorientable}
The GKM graph of a GKM manifold is orientable.
\end{prop}

As we will be interested specifically in sphere bundles over spheres, we need to review the notion of a GKM fibration. It was introduced in \cite{GuilleminSabatiniZara} although the definition below differs slightly in that we consider unsigned GKM graphs and do not fix connections as part of the data of a GKM graph (cf.\ also \cite{GKMflagbundles}). It was shown in \cite[Proposition 3.7]{MR4634089} that indeed a $T$-equivariant fiber bundle whose total space and base are GKM manifolds induces the structure of a GKM fibration on the level of GKM graphs in the sense of the following

\begin{defn}\label{defn:GKMfibration}
\begin{enumerate}
\item A \emph{morphism} $\Gamma\rightarrow B$ between two graphs is a map $\pi\colon V(\Gamma)\rightarrow V(B)$ as well as a partially defined map $\pi\colon E(\Gamma)\rightarrow E(B)$ which associates to every edge $e\in E(\Gamma)$ with $\pi(i(e))\neq \pi(t(e))$ an edge $\pi(e)$ between $\pi(i(e))$ and $ \pi(t(e))$. Edges with this property are called \emph{horizontal}. The others are called \emph{vertical}. Horizontal edges emanating from $v\in V(\Gamma)$ are denoted by $H_v$. The morphism is called a \emph{graph fibration} if for all $v\in V(\Gamma)$ the map $H_v\rightarrow B_{\pi(v)}$ is an isomorphism.

\item A \emph{GKM fibration} $(\Gamma,\alpha)\rightarrow (B,\alpha_B)$ between two GKM graphs consists of a graph fibration $\Gamma\rightarrow B$ that is compatible with choices of connections $\nabla$, $\nabla^B$ on $\Gamma$, $B$ in the sense that
\begin{enumerate}[label=(\alph*)]
\item $\alpha (\pi(e))= \alpha_B(e)$ for all horizontal edges $e$.
\item For every $e\in E(\Gamma)$ the map $\nabla_e$ respects the decomposition in horizontal and vertical edges.
\item $\nabla$ covers $\nabla_B$ in the sense that for any $e,f\in H_v$, $v\in V(\Gamma)$, we have $\pi(\nabla_e(f))=\nabla^B_{\pi(e)}(\pi(f))$
\end{enumerate}
\end{enumerate}
\end{defn}

\subsection{Cohomology}\label{SubSec: }
All cohomologies are understood with integer
coefficients.
For a torus $T$ let $T \to ET \to BT$ be the universal $T$-bundle and assume that $T$ acts
on a manifold $M$. The \emph{Borel construction} is defined as
\[
	M_{T} := ET \times_{T} M,
\]
i.e. $M_{T}$ is the associated $M$-bundle $\pi \colon M_{T} \to BT$ to $ET \to BT$. The
singular cohomology $H^{\ast}(M_{T})$ of $M_{T}$ is denoted by $H_{T}^{\ast}(M)$ and is
called the \emph{equivariant cohomology} of the $T$-action on $M$. The projection $\pi$
induces by pull-back an $H^{\ast}(BT)$-module structure on $H^{\ast}_{T}(M)$. The
cohomology ring $R:=H^{\ast}(BT)$ is isomorphic to the polynomial ring $\mathbb{Z}[x_{1},
\ldots, x_{k}]$, where $x_{i}$ are of degree $2$ and $k$ is the rank of $T$. To see this, recall the isomorphisms 
\[
BT \cong BS^{1} \times \ldots \times BS^{1} \cong \mathbb C\mathbb P^{\infty} \times
\ldots \times \mathbb C\mathbb P^{\infty}
\]
and $H^{\ast}(\mathbb C\mathbb P^{\infty}) \cong \mathbb{Z}[x]$, where $x$ is of degree
$2$. In particular if the manifold is a point $M = \ast$, then $H_{T}^{\ast}(\ast) =
H^{\ast}(BT)=R$.

\begin{rem}\label{rem: weights and R}
Let $\lambda \in \operatorname{Hom}(T,S^{1})$ be a character which defines a $T$-representation on $\mathbb{C}$. Consider the equivariant map $\mathbb{C} \to \{p\}$, where
$\{p\}$ is equipped with the trivial $T$-action. This induces a complex vector bundle
$L_{\lambda} = ET \times_{T} \mathbb{C} \to ET \times_{T} \{p\}$ and a first Chern class
$c_{1}(L_{\lambda}) \in H^{2}_{T}(\{p\}) = H^{2}(BT)$. This defines an isomorphism from
$\operatorname{Hom}(T,S^{1})$ to $H^{2}(BT)$. By abuse of notation, we denote the image of the
characters in $H^{2}(BT)$ again by $\lambda$.
\end{rem}

We say the $T$-action on $M$ is \emph{equivariantly formal (over $\ZZ$)} if $H_{T}^{\ast}(M)$ is free as
an $R$-module. It turns out that GKM manifolds are equivariantly formal:

\begin{thm}[{\cite[Proposition 2.22]{MR4840842}}]\label{T: equivariantly formal}
Let $T$ act on a compact manifold $M$ with isolated fixed points. Then the following conditions are equivalent. 
\begin{enumerate}[label=(\alph*)]
	\item The action of $T$ on $M$ is equivariantly formal.
	\item $H^{\mathrm{odd}}(M) = 0$.
\end{enumerate}
\end{thm}

For an equivariantly formal action the
ordinary cohomology can be computed from its equivariant cohomology via the isomorphism $H^*(M)\cong H^*_T(M)/R^{>0}\cdot H^*_T(M)$ induced by the restriction map, see \cite[Theorem
1.1]{MR2308029}.

Furthermore, if the action is equivariantly formal, it follows, under certain assumptions, that the
equivariant cohomology is determined by its equivariant
$1$-skeleton. This is known as the \emph{Chang-Skjelbred Lemma}, see \cite{MR375357,
MR2308029}. For the purpose of this paper it is enough to know

\begin{prop}\label{P: injectivity equivariant cohomology}
Consider a GKM $T$-action on $M$.
Let $i \colon M^{T} \to M$ be the equivariant inclusion of the fixed point set. Then
\[
	i^{\ast} \colon H^{\ast}_{T}(M) \longrightarrow H^{\ast}_{T}(M^{T}) 
\]
is injective.
\end{prop}

\begin{proof}
By Borel localization \cite[Theorem (3.2.6)]{AlldayPuppe} we have that the kernel of $i^{\ast}$ consists
of $H^{\ast}(BT)$-torsion, but since the action is equivariantly formal, the equivariant
cohomology $H_{T}^{\ast}(M)$ is $H^{\ast}(BT)$-torsion free, since it is free as an
$H^{\ast}(BT)$-module.
\end{proof}

\subsection{Equivariant characteristic classes}\label{SubSec: Equivariant Char Classes}

Suppose $(M,T)$ is a GKM manifold and $\pi \colon V \to M$ a vector bundle over $M$
(either real or complex). We
say that $\pi \colon V \to M$ is a \emph{$T$-equivariant vector bundle} if $V$ is equipped with a $T$-action such that $\pi$ is a
$T$-equivariant map and the $T$-action on $V$ is fiberwise linear. In this case we may apply the Borel construction to this bundle and
obtain the vector bundle
\[
	V_{T}  =  ET \times_{T} V \longrightarrow ET \times_{T} M = M_{T}.
\]
\begin{defn}\label{D: Equivariant CC}
Let $\kappa$ be a characteristic class for real or complex vector bundles. Suppose $V
\to M$ is a $T$-equivariant vector bundle. The \emph{equivariant characteristic class $\kappa^{T}(V)$ of $V$} is defined as
\[
	\kappa^{T}(V) := \kappa(V_{T}) \in H_{T}^{\ast}(M).
\]
\end{defn}

The tangent bundle $TM \to M$ of $M$ is a $T$-equivariant bundle over $M$ by taking
differentials. Furthermore, the inclusion $i \colon M^{T} \to M$ induces an inclusion $i \colon
(M^{T})_{T} \to M_{T}$ and pulling back $TM_{T}$ by $i$ gives the bundle
$i^{\ast}((TM)_{T}) \to (M^{T})_{T}$. By the naturality property of characteristic classes and
Proposition \ref{P: injectivity equivariant cohomology} the equivariant characteristic
classes are uniquely determined by $i^{\ast}(\kappa^{T}(M))$.

Now choose an invariant complex structure on $T_pM$, where $p \in M^{T}$. Thus, $T_pM$ 
decomposes into a sum of weight spaces $T_pM = \oplus_{\alpha} V_{\alpha}$, which
implies that 
\[
V_{p}:=i^{\ast}((TM)_{T})|_{\{p\}_T} = \bigoplus_{\alpha} L_{\alpha},
\]
where $L_{\alpha} = ET \times_{T} V_{\alpha}$ is a complex line bundle over $BT = \{p\}_{T}$.
Now let $p$ denote the \emph{total Pontryagin class} of a real vector bundle and let $c$
denote the \emph{total Chern class} of a complex vector bundle. The Chern roots of $V_{p}$
are given by the weights $\alpha_{1}, \ldots, \alpha_{n} \in H^{2}(BT)$ (cf. Remark
\ref{rem: weights and R}) of the isotropy representation in $p \in M^{T}$, from which the Pontryagin classes $p_i^T(V_p)\in H_T^*(\{p\})\cong R$ can be
computed as

\begin{align*}
	p_{i}^{T}(V_{p}) &= (-1)^{i} c^{T}_{2i}(V_{p} \oplus \overline{V_{p}})\\
									 &= (-1)^{i}\sigma_{i}(-\alpha_{1}^{2},\ldots, -\alpha_{n}^{2} )\\
									 &= \sigma_{i}(\alpha_{1}^{2}, \ldots, \alpha_{n}^{2})
\end{align*}
where $\sigma_{i}$ is the $i$-th elementary polynomial and $\overline{V_{p}}$ denotes the
conjugate bundle to $V_{p}$. Note that the equivariant Pontryagin class does not depend
on the choice of the invariant complex structure on $TM_{p}$.
Finally, denoting by $1_p$ the element of $H_T^*(M^T)$ which is $1$ at $p$ and $0$ at all other fixed points, we obtain
\begin{prop}\label{P: Computation PC}
Let $(M,T)$ be a GKM manifold. The total equivariant Pontryagin class of $(M,T)$ regarded as an
element in $H^{\ast}_{T}(M^{T})$ is given by
\[
	p^{T}(M) = \sum_{p \in M^{T}} 1_p\cdot\prod_{\alpha}(1+\alpha^{2})
\]
where the product runs over all weights, which occur in the isotropy representation in $p
\in M^{T}$.
\end{prop}

\subsection{Linear sphere bundles}\label{SubSec: Sphere bundles}
Consider the isomorphism classes of oriented Euclidean vector bundles of rank
$k$ over $S^{m}$ and denote it by $\mathrm{V}_{+}(k,m)$. A \emph{linear sphere bundle}
over $S^{m}$ is the sphere bundle of an element of $\mathrm{V}_{+}(k,m)$. The orientation on the vector
bundle induces an orientation on the fiber spheres, determined by the outward-pointing
normal to the sphere. Thus, every linear sphere bundle is canonically oriented. We say two
linear sphere bundles are \emph{isomorphic} if their corresponding vector bundles
represent the same element in $\mathrm{V}_{+}(k,m)$.

For every map $\chi \colon S^{m-1} \to \mathrm{SO}(k)$ one may define the vector bundle
\begin{equation}\label{eq:clutchingbdl}
	E_{\chi} = (D^{m}_{+} \times \mathbb{R}^{k}) \cup_{\chi} (D^{m}_{-} \times \mathbb{R}^{k}) 
\end{equation}
where we identified $(q,v)$ with $(q,\chi(q)(v))$.
Clearly $E_{\chi} \in \mathrm{V}_{+}(k,m)$. Maps like $\chi$ are called \emph{clutching
maps}. From \cite[Proposition 1.14]{hatcherVBandKTheory} we have

\begin{prop}\label{P: homotopy and vb}
The map 
\[
	\pi_{m-1}(\mathrm{SO}(k)) \longrightarrow \mathrm{V}_{+}(k,m), \quad [\chi] \mapsto
	[E_{\chi}]
\]
is a bijection. Thus, linear sphere bundles over $S^{m}$ with fiber $S^{k}$ are
classified, up to isomorphism, by
$\pi_{m-1}(\mathrm{SO}(k))$.
\end{prop}

Next, we would like to understand how Pontryagin classes of total spaces of sphere bundles over
spheres are related to its clutching maps. We use the idea of \cite{MR0082103} and
flesh out the details.

\begin{rem}\label{rem: first PC for sphere bundle}
Let $\pi \colon E \to S^{m}$ be a real vector bundle over a sphere $S^{m}$ of rank $k$. Denote by
$S^{k-1} \to M \stackrel{\pi}{\to} S^{m}$ the induced sphere bundle of $E \to S^{m}$. Then
we have the relation
\[
	p(M) = \pi^{\ast}(p(E)) \in \pi^{\ast}(H^{\ast}(S^{m}))
\]
where $p$ denotes the total Pontryagin class. For that note first that
\[
	TM \oplus \varepsilon^{1} \cong \pi^{\ast}(TS^{m}) \oplus \pi^{\ast}(E)
\]
where $\varepsilon^{r}$ denotes the trivial bundle of rank $r$. By the Whitney sum formula
for Pontryagin classes and the fact the $S^{m}$ has trivial stable tangent bundle, we see
that $p(S^{m})= 1$. Using again the Whitney sum
formula and the naturality of Pontryagin classes we obtain
\[
	p(M) = p(TM \oplus \varepsilon^{1}) = \pi^{\ast}(p(S^{m})) \smile \pi^{\ast}(p(E)) =
	\pi^{\ast}(p(E)).
\]
\end{rem}

Suppose $m \neq k$. For $\varphi \in \pi_{m-1} (\mathrm{SO}(k))$ denote by
$E_{\varphi}$ the associated vector bundle as in \eqref{eq:clutchingbdl},  by $M_{\varphi}$ the total space of its
sphere bundle, and by $\pi_{\varphi} \colon M_{\varphi} \to S^{m}$ the projection.

\begin{lem}\label{L: Milnor linearity}
For $m \equiv 0 \mod 4$ the map
\[
	\beta \colon \pi_{m-1}(\mathrm{SO}(k)) \longrightarrow H^{\ast}(S^{m}), \quad \chi \mapsto
	p_{\frac{m}{4}}(E_{\chi})
\]
is a homomorphism.
\end{lem}

\begin{proof}
Let $\chi_{1}, \chi_{2} \in \pi_{m-1}(\mathrm{SO}(k))$. Then $\chi_{1} +
\chi_2$ is represented by the concatenation of the maps
\[
	S^{m-1} \stackrel{c}{\longrightarrow} S^{m-1} \vee S^{m-1} \stackrel{\chi_{1} \vee \chi_{2}}{\longrightarrow} \mathrm{SO}(k) 
\]
(by abuse of notation we do not distinguish between the homotopy class $\chi_{i}$ and
its
representatives $\chi_{i} \colon S^{m-1} \to \mathrm{SO}(k)$), where $c$ is the
collapse map of the equator $S^{m-2} \subset S^{m-1}$.
Define by a clutching construction the vector bundle $E_{\chi_{1} \vee \chi_{2}}$
over $S^{m} \vee S^{m}$ by $\chi_{1} \vee \chi_{2}$ so that
\[
	E_{\chi_{1} + \chi_{2}} =c^{\ast}(E_{\chi_{1} \vee \chi_{2}}) .
\]
By using the inclusion $S^{m} \to S^{m} \vee S^{m}$ into each factor one sees that
\[
p_{r}(E_{\chi_{1} \vee \chi_{2}}) = (p_{r}(E_{\chi_{1}}),
p_{r}(E_{\chi_{2}})) \in H^{4r}(S^{m} \vee S^{m}) = H^{4r}(S^{m}) \oplus
H^{4r}(S^{m})
\] with $m = 4r$. Therefore, by naturality it follows 
\[
p_{r}(E_{\chi_{1} + \chi_{2}})
= c^{\ast}(p_{r}(E_{\chi_{1} \vee \chi_{2}})) = p_{r}(E_{\chi_{1}}) +
p_{r}(E_{\chi_{2}}).
\]
\end{proof}

We assume now that $k=5$ and $m=4$, i.e., consider linear $S^4$-bundles
over $S^4$. By Proposition \ref{P: homotopy and vb}, $V_+(5,4)$ is in one-to-one
correspondence to $\pi_{3}(\mathrm{SO}(5))$, which is isomorphic to $\mathbb{Z}$. Thus the map $\beta\colon \pi_3(\SO(5))\rightarrow H^4(S^4)$ is determined by its value on a generator, which we now compute. To this end we use

\begin{lem}\label{L: iso on pi_3}
Consider the embedding of $\mathrm{Sp}(1) \cong \mathrm{SU}(2) \subset \mathrm{U}(2)
\subset \mathrm{SO}(4) \subset \mathrm{SO}(5)$. This inclusion $\mathrm{Sp}(1)
\hookrightarrow \mathrm{SO}(5)$ induces an isomorphism between the third homotopy groups.
\end{lem}

\begin{proof}
Since the representation of $\mathrm{Sp}(1) \cong \mathrm{SU}(2)$ through $\mathrm{SO}(4)$ is the
multiplication of a unit quaternion from the left on $\mathbb H \cong \mathbb{R}^{4}$, we
see that the lift of $\mathrm{Sp(1)}$ into $\mathrm{Spin}(4) \cong \mathrm{Sp}(1) \times
\mathrm{Sp}(1)$ is given by $q \mapsto (q,1)$. Furthermore,
the spin representation of $\mathrm{Spin}(5) \cong \mathrm{Sp}(2)$ is the standard
representation of $\mathrm{Sp}(2)$ on $\mathbb H^{2}$. From \cite[Proposition 5.1]{MR1701598}
it follows that if the $\mathrm{Spin}(5) \cong \mathrm{Sp}(2)$ representation is restricted to
$\mathrm{Spin}(4)$, we obtain the embedding
\[
 \mathrm{Spin}(4)\cong \mathrm{Sp}(1) \times \mathrm{Sp}(1) \to \mathrm{Spin}(5) \cong
 \mathrm{Sp}(2), \quad (q,p) \mapsto 
 \begin{pmatrix}
q & 0 \\
0 & p
 \end{pmatrix}.
\]
In total, the embedding $\mathrm{Sp}(1) \to \mathrm{SO}(5)$ fits into the diagram
\[
	\begin{tikzcd}
	& \mathrm{Spin}(4) \arrow{r}\arrow{d} & \mathrm{Sp}(2)\arrow{d} \\
		\mathrm{Sp}(1)\arrow{ru}\arrow{r} & \mathrm{SO}(4) \arrow{r} & \mathrm{SO}(5)
	\end{tikzcd} 
\]
where the upper arrows are the maps
\[
	q \mapsto (q,1) \mapsto 
	\begin{pmatrix}
q & 0  \\
0 & 1
	\end{pmatrix}.
\]
The quotient $\mathrm{Sp}(2)/\mathrm{Sp(1)}$ is diffeomorphic to $S^{7}$ and therefore by the long
exact sequence for homotopy groups of the fibration $\mathrm{Sp}(1) \to \mathrm{Sp}(2) \to
S^{7}$ the embedding $\mathrm{Sp}(1) \hookrightarrow \mathrm{Sp}(2)$ induces an
isomorphism on $\pi_{3}$. Since the projection $\mathrm{Sp}(2) \to \mathrm{SO}(5)$ is also
an isomorphism on $\pi_{3}$ the claim follows.
\end{proof}

Let $\chi_1$ denote the embedding $S^3\cong
	\mathrm{Sp}(1)\rightarrow\SO(5)$ which yields a generator of $\pi_3(\SO(5))$ by Lemma
	\ref{L: iso on pi_3}. Since $\chi_1$ takes values in $\SO(4)\subset \SO(5)$ the bundle
	$E_{\chi_1}$ is given as the Whitney sum of a rank $4$ bundle with a trivial rank $1$
	bundle. Hence, by the Whitney sum formula $E_{\chi_{1}}$ has the same
	Pontryagin classes as the rank $4$ bundle. This  rank $4$ bundle with clutching function the identification
	$S^3\cong \mathrm{Sp}(1)$ is the tautological bundle $H \to \mathbb{HP}^{1} \cong
S^{4}$. The sphere bundle of $H \to S^{4}$ is diffeomorphic to
$S^{7}$ and by the Gysin sequence we see that the Euler class of $H$ is a generator of
$H^{4}(S^{4})$, which we denote by $\iota$. From \cite[Lemma 20.9]{MR0440554} we have
$p_1(H)= - 2 \cdot \iota$.
Hence $\beta(\chi_1)=- 2\cdot\iota$. Now by Lemma \ref{L: Milnor linearity} and Remark
\ref{rem: first PC for sphere bundle} we obtain

\begin{cor}\label{C: clutching function pontryagin}
For any $l\in \ZZ$, the Pontryagin class $p_{1}(M_{l\cdot \chi_1})$ is equal to
$- 2l \cdot \pi^{\ast}(\iota)$, where $\pi:M_{l\cdot \chi_1}\to S^4$ is the projection.
\end{cor}

\subsection{Cohomology rings of linear $S^{4}$-bundles over $S^{4}$}

 As above $H \to S^{4}$ denotes the tautological bundle over $\mathbb H\mathbb P^{1} \cong
 S^{4}$. For $l \in \mathbb{Z}$ let $f_{l} \colon S^{4} \to S^{4}$ be a map of
 degree $l$. Set $H_{l} :=f^{\ast}_{l}(H)$ and $E_{l} := H_{l} \oplus \varepsilon^{1}$
 (where $\varepsilon^{m}$ is the trivial rank $m$ bundle over $S^{4}$). We endow $H_{l}$
 with an inner product and we choose the product metric on $E_{l}$.
\begin{lem}\label{L: identification with k-tautological bundles}
We have $E_{l} \cong E_{l \cdot \chi_{1}}$ for all $l \in \mathbb{Z}$.
\end{lem}
\begin{proof}
Using the naturality of Pontryagin classes we have
\[
	p_{1}(E_{l}) = p_{1}(H_{l}) = f^{\ast}_{l}(p_{1}(H)) = - 2l \cdot \pi^{\ast}(\iota).
\]
Thus, with Corollary \ref{C: clutching function pontryagin}, the clutching function of
$E_{l}$ is up to homotopy given by $l \cdot \chi_{1}$.
\end{proof}

For a Euclidean vector bundle $E$ denote by $S(E)$ its sphere bundle. Then,
\[
	S(E_{l}) = \{(v,t) \in H_{l} \oplus \varepsilon^{1} : |v|^{2} + t^{2} = 1\}.
\]
and 
\[
	S(E_{l}) = D_{+}(H_{l}) \cup_{S(H_{l})} D_{-}(H_{l})
\]
where 
\[
D_{\pm}(H_{l}) = \{ (v,t) \in H_{l} \oplus \varepsilon^{1} : |v|^{2} +t^{2} = 1,\, \pm t
\geq 0\} .
\]
We abbreviate $S_{l} = S(E_{l})$ and $D_{\pm, l} = D_{\pm }(H_{l})$.
Note that $D_{\pm,l}$ 
	is isomorphic to the disc bundle of $H_{l}$. Consider the pair
	$(S_{l}, D_{-,l})$. The quotient $S_{l}/D_{-,l}$ is the Thom space $\mathrm{Th}(H_{l})$ of
	$H_{l}$. The long exact sequence of the pair above then yields
	\[
		\ldots \longrightarrow H^{i-1}(D_{+,l}) \stackrel{\delta}{\longrightarrow}
		\widetilde{H}^{i}(\mathrm{Th}(H_{l})) \longrightarrow H^{i}(S_{l}) \longrightarrow
		H^{i}(D_{+,l}) \longrightarrow \ldots,
	\]
where $\delta$ is the
connecting homomorphism. Clearly $D_{\pm,l}$ is homotopy equivalent to $S^{4}$. The Thom isomorphism gives that the reduced cohomology of
$\mathrm{Th}(H_{l})$ is zero except for $i=4$ and $i=8$. In particular, we have the short
exact sequence
\[
	0 \longrightarrow \widetilde{H}^{4}(\mathrm{Th}(H_{l})) \longrightarrow H^{4}(S_{l})
	\longrightarrow H^{4}(S^{4}) \longrightarrow 0.
\]
The group $\widetilde{H}^{4}(\mathrm{Th}(H_l))$ is generated by the Thom class $U_{H_{l}}$
and therefore $H^{4}(S_{l}) \cong \mathbb{Z} \oplus \mathbb{Z}$, where one summand is
generated by $x:=U_{H_{l}}$ and the other by $y:=\pi^{\ast}(\iota)$.

\begin{prop}\label{P: cup products on H^4}
In $H^{\ast}(S_{l})$ we have
\[
	x^{2} = l \cdot x \smile y \quad \text{and} \quad y^{2} = 0. 
\]
Thus, the ring structure is given by
\[
	\mathbb{Z}[x,y] /( x^{2}-l (x \smile y), y^{2} )
\]

\end{prop}
\begin{proof}
	The relation $y^{2} =0$ is obvious and by Poincaré duality the class $x \smile y$ is a
	generator of $H^{8}(S_{l})$. It is known \cite[p.\ 99]{MR0440554} that 
	\[
		x^{2} = \pi^{\ast}(e(H_{l})) \smile x.
	\]
Furthermore $\pi^{\ast}(e(H_{l})) = l \cdot \pi^{\ast}(\iota) = l \cdot y$.	
\end{proof}

\begin{cor}\label{T: cohomology rings of S_l}
The cohomology ring of $S_{l}$ is given by
\[
	\mathbb{Z}[x,y]/ ( x^{2}, y^{2} )
\]
if $l$ is even and
\[
	\mathbb{Z}[x,y]/ ( x^{2} - x \smile y, y^{2} )
\]
if $l$ is odd.
\end{cor}

\begin{proof}
For $\bar x = x + ky$ we see that $\bar x^{2} =(l+2k)x\smile y$, where $k \in
\mathbb{Z}$. Hence, we may add any even number to $l$ and obtain isomorphic rings.
\end{proof}

\section{Combinatorial classification}\label{sec:combinatorial}

In this section we classify possible GKM graphs $(\Gamma,\alpha)$ of GKM manifolds $M$ such that there is a $T^2$-equivariant fiber bundle $M\rightarrow S^4$ with fibers diffeomorphic to $S^4$, where the action on the basis $S^4$ is the standard GKM action. In order to distinguish the cases in the classification we make a certain choice of connection. The latter is not always canonical and as a result some overlap in the cases arises when not fixing the connection. We make this precise in Remark \ref{rem: overlap}.

As proved in \cite[Proposition 3.7]{MR4634089} $(\Gamma,\alpha)$ is a GKM fibration (see Definition \ref{defn:GKMfibration}) over the GKM graph of the base $S^4$. From the definition of a GKM fibration it follows that for suitable $a,b,\ldots,h$ the GKM graph $(\Gamma,\alpha)$ is isomorphic to

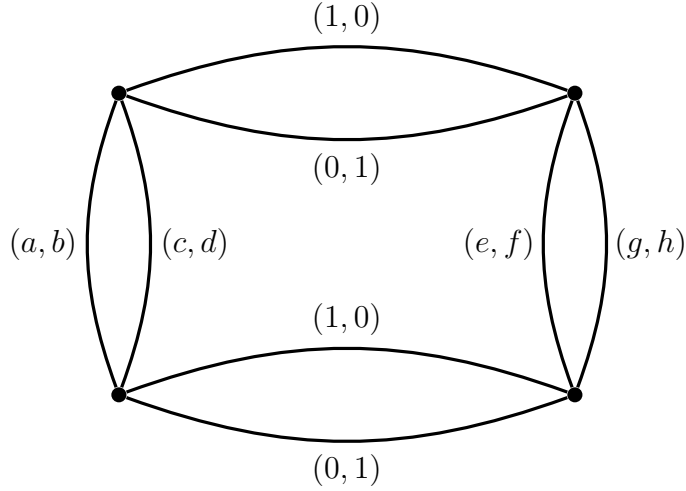
\begin{center}
		\begin{tikzpicture}

			\node (a) at (0,0)[circle,fill,inner sep=2pt] {};
			\node (b) at (6,0)[circle,fill,inner sep=2pt]{};
			\node (c) at (0,4)[circle,fill,inner sep=2pt]{};	
			\node (d) at (6,4)[circle,fill,inner sep=2pt]{};			
			
			\node at (-1,2) {$(a,b)$};
			\node at (+1,2) {$(c,d)$};
			\node at (5,2) {$(e,f)$};
			\node at (7,2) {$(g,h)$};
			\node at (3,5) {$(1,0)$};
			\node at (3,3) {$(0,1)$};
			\node at (3,1) {$(1,0)$};
			\node at (3,-1) {$(0,1)$};

			\draw [very thick](a) to[in=160, out=20] (b);
			\draw [very thick](a) to[in=200, out=-20] (b);
			\draw [very thick](c) to[in=160, out=20] (d);
			\draw [very thick](c) to[in=200, out=-20] (d);
			\draw [very thick](a) to[in=250, out=110] (c);
			\draw [very thick](a) to[in=290, out=70] (c);
			\draw [very thick](b) to[in=250, out=110] (d);
			\draw [very thick](b) to[in=290, out=70] (d);
			
		\end{tikzpicture}
	\end{center}
which we call \emph{product type} or
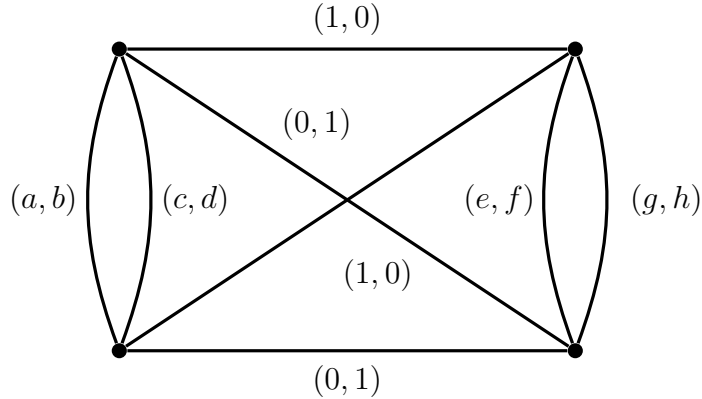
\begin{center}
		\begin{tikzpicture}

			\node (a) at (0,0)[circle,fill,inner sep=2pt] {};
			\node (b) at (6,0)[circle,fill,inner sep=2pt]{};
			\node (c) at (0,4)[circle,fill,inner sep=2pt]{};	
			\node (d) at (6,4)[circle,fill,inner sep=2pt]{};			
			
			\node at (-1,2) {$(a,b)$};
			\node at (+1,2) {$(c,d)$};
			\node at (5,2) {$(e,f)$};
			\node at (7.2,2) {$(g,h)$};
			\node at (3,4.4) {$(1,0)$};
			\node at (2.6,3) {$(0,1)$};
			\node at (3.4,1) {$(1,0)$};
			\node at (3,-.4) {$(0,1)$};

			\draw [very thick](a) to (b);
			\draw [very thick](a) to (d);
			\draw [very thick](c) to (d);
			\draw [very thick](c) to (b);
			\draw [very thick](a) to[in=250, out=110] (c);
			\draw [very thick](a) to[in=290, out=70] (c);
			\draw [very thick](b) to[in=250, out=110] (d);
			\draw [very thick](b) to[in=290, out=70] (d);
			
		\end{tikzpicture}
	\end{center}
which we call \emph{twisted type}. In our enumeration of the possible cases we will indicate this by the respective letter {\bf P} or {\bf T}. As we consider unsigned graphs (see Remark \ref{rem:signedunsigned}), labels are considered up to sign. The vertical two edge graphs correspond to the fibers over the fixed points. We call those edges vertical, the others horizontal.

\begin{lem}\label{lem: goodcon}
One can choose a connection $\nabla$ on $\Gamma$ as in Definition \ref{defn:GKMfibration} satisfying additionally that it is invariant under graph automorphism that swaps the vertices in each fiber while preserving the labels on edges.
\end{lem}
\begin{proof}
We already know that we have a connection on $\Gamma$ satisfying the conditions of Definition \ref{defn:GKMfibration}. If it is not compatible with the swap automorphism, we may keep it on the two horizontal edges $e_1,e_2$ emerging from one of the vertices in one fiber, and modify it on the other two by replacing it with the connection along $e_1$ respectively $e_2$, conjugated by the swap automorphism.
\end{proof}

We fix a connection $\nabla$ as in the above lemma. The transports along the horizontal $(1,0)$ and $(0,1)$ edges can either agree or disagree on the vertical edges. We distinguish these cases by the letters {\bf A} (agree) and {\bf D} (disagree). By the choice of connection this does not depend on which horizontal edges we consider. We now fix signs for the weights in the fibers such that the congruences for the transport along the $(1,0)$-edges hold with positive sign. Now for these fixed sign choices the congruences along the $(0,1)$-edge hold with a certain unique sign. In case both are positive (resp.\ negative) we denote the case by $++$ (resp.\ $--$). In case signs are mixed we denote the case as $+-$. So far we have associated a GKM fibration together with a choice of connection as in \ref{lem: goodcon} and a choice of signs for the weights uniquely to one of $12$ cases (e.g.\ of the form \textbf{PA}$++$).

\begin{rem}\label{rem: labels explicit}
Within each case we can give an exhaustive list of possible axial functions up to isomorphism, parametrized over the labels $(a,b)$, $(c,d)$ in one fiber (see the picture above): after potentially swapping $(e,f)$ and $(g,h)$ we may assume that the transport along the $(1,0)$ horizontal edges maps the $(a,b)$ edge to the $(e,f)$ edge and the $(c,d)$ edge to the $(g,h)$ edge. With sign choices as above the congruences
\[(a,b)\equiv \pm(e,f)\mod (1,0),\qquad (c,d)\equiv \pm(g,h)\mod (1,0)\]
hold with positive sign, implying $b=f$, $d=h$. Now the congruences along the $(0,1)$-edges
yield
$(a,b)=(\pm e,f)$, $(c,d)=(\pm g,h)$ in cases \textbf{PA}, \textbf{TA} and
$(a,b)=(\pm g,f)$, $(c,d)=(\pm e,h)$ in cases \textbf{PD}, \textbf{TD} with the signs provided by the case distinction $++$, $+-$, $--$ (in the case $+-$ we assume that the first equality holds with positive sign and the second to hold with negative sign; otherwise swap the labels in both fibers). See Section \ref{sec:construction} for pictures of the occurring (orientable) GKM graphs.

Finally we observe that the condition of pairwise linear independence of adjacent weights forces some restrictions on which $a,b,c,d$ can occur. In every case $a,b,c,d\in \ZZ\backslash\{0\}$ due to linear independence between vertical and horizontal weights. Furthermore $ad-cb\neq 0$ due to linear independence of the vertical weights in the left hand fiber. Furthermore we can get an additional condition for linear independence of horizontal weights in the right hand fiber depending on the case:
\begin{center}
\begin{tabular}{c|c}
	case &  condition\\ \hline
	$\mathbf{A}++$, $\mathbf{A}--$ &  \\
	$\mathbf{A}+-$  & $ad+bc\neq 0$\\
	$\mathbf{D}++$, $\mathbf{D}--$ &  $cd-ab\neq 0$\\
	$\mathbf{D}+-$  & $cd+ab\neq 0$
	\end{tabular}
	\end{center}
Any choice of $a,b,c,d$ subject to the above conditions produces a GKM fibration of the respective type.
\end{rem}

With the above analysis of the labels we return to see whether the previous assignment of a case depends on the sign choices: In case \textbf{A} the result of the the $\{++,--,+-\}$ component is independent of the initial choice of signs for the weights in the fibers: indeed, suppose that $(a,b)=(\epsilon_1 e,f)$, $(c,d)=(\epsilon_2 g,h)$ for $\epsilon_i\in\{\pm 1\}$ as in Remark \ref{rem: labels explicit}. Then a different sign choice corresponds to replacing $(a,b),(e,f)$ by $(-a,-b),(-e,-f)$ (or analogously doing the same for the pair $(c,d),(g,h)$). Using the above equations, the $(0,1)$-congruences still read
\begin{align*}
(-a,-b)&= (-\epsilon_1e,-f)\equiv \epsilon_1 (-e,-f)\mod (0,1)\\
(c,d) &= (\epsilon_2 g,h)\equiv \epsilon_2(g,h)\mod (0,1)
\end{align*}
with the same signs $\epsilon_1,\epsilon_2$. The same happens when swapping the sign of the other pair of weights.

However in case \textbf{D} choosing a different sign for one weight in each fiber weights flips the signs of both congruences along the $(0,1)$-edges: Taking the equations $(a,b)=(\epsilon_1g,f)$ and $(c,d)=(\epsilon_2e,h)$ from Remark \ref{rem: labels explicit} and flipping the signs of the weights $(a,b)$, $(e,f)$, the congruences
 \begin{align*}
 (-a,-b)&=(-\epsilon_1g,-f)\equiv -\epsilon_1(g,h)\mod(0,1)\\
 (c,d)&=(\epsilon_2e,h)\equiv -\epsilon_2(-e,-f)\mod(0,1)
 \end{align*}
now hold with flipped signs. The same is observed when flipping the signs for the pair $(c,d),(g,h)$.
 
 Hence the cases \textbf{D$++$} and \textbf{D$--$} can be interchanged by choosing different signs for the weights. Henceforth we do not distinguish these cases, usually denoting both by \textbf{D}$++$. We arrive at a classification which takes as an input a GKM fibration with a choice of connection as in Lemma \ref{lem: goodcon} and associates it uniquely to one of the cases \textbf{A$++$}, \textbf{A$--$}, \textbf{A$+-$}, \textbf{D$++$}, and \textbf{D$+-$}, each of which can occur in combination with underlying graph structure \textbf{P} or \textbf{T}.

\begin{rem}\label{rem: overlap}
The choice of connection is not always unique: If $a=\pm c$, the connection along the horizontal $(0,1)$-edges can be freely chosen to be any bijection between the fiber edges. Similarly if $b=\pm d$, the same holds for the $(1,0)$-edges. Hence if either holds, the GKM graph can be endowed with different connections as in Lemma \ref{lem: goodcon} such that it can be associated to case \textbf{A} as well as \textbf{D}. Note however that once the connection is fixed, the sign component in the case distinction are uniquely determined as described above.
\end{rem}

As Proposition \ref{prop:GKMgraphorientable} states that the GKM graph of a GKM manifold is orientable, we need to understand which of the total spaces of the GKM fibrations under consideration are orientable.

\begin{prop}\label{prop:combinatorialclassification} Consider a GKM $T^2$-fibration $\Gamma\rightarrow B$, whose fiber is a biangle, where $B$ is the graph of the standard action on $S^4$ and choose a connection as in Lemma \ref{lem: goodcon}. If $\Gamma$ belongs to one of the types
\[
{\mathbf{PA}}++,\quad {\mathbf{PD}}++,\quad {\mathbf{PA}}--,\quad {\mathbf{TA}}+-, \quad \textrm{or} \quad {\mathbf{TD}}+-.
\]
then it is orientable. If it belongs to the remaining types
\[
{\mathbf{PA}}+-,\quad {\mathbf{PD}}+-,\quad {\mathbf{TA}}++,\quad {\mathbf{TA}}--
,\quad \textrm{or} \quad {\mathbf{TD}}++
\]
then it is not orientable.
\end{prop}

\begin{proof}
We choose signs for the weights as prescribed by the labels above. Then for any vertical edge $e$, $\eta(e)=-1$, where $\eta:E(\Gamma)\to \{\pm 1\}$ is the map used in the definition of orientability, see Definition \ref{defn:orientable}. Also, any horizontal edge $e$ labelled $(1,0)$ satisfies $\eta(e)=-1$, while any horizontal edge $e$ labelled $(0,1)$ satisfies $\eta(e)=-1$ in case of a fibration of type $++$ or $--$, and $\eta(e)=1$ in case of a fibration of type $+-$. 
Hence, the graph of a fibration of product type is orientable if and only if it is of type $++$ or $--$: any closed edge path consists of an even number of edges. The graph of a fibration of twisted type is orientable if and only if it is of type $+-$.
\end{proof}

\section{Construction of the bundles} \label{sec:construction}

\begin{rem}\label{rem: conventions}
We fix some conventions for this section.
\begin{enumerate}
\item We identify $\RR^4\cong \CC^2$ via $(a,b,c,d)\mapsto (a+bi,c+di)$. This identification fixes an embedding $\U(2)\rightarrow \SO(4)$.
\item We identify $\CC^2\cong \HH$ via $(u,v)\mapsto u+jv$. When identifying $\RR^4\cong \HH$ we use the composition of $(i)$ and $(ii)$.
\item When translating between $\SU(2)$ and unit quaternions we identify $\left(\begin{matrix}z & -\overline{w}\\ w & \overline{z} \end{matrix}\right)$ with the unit quaternion $z+jw$. Note that the standard action of $\SU(2)$ on $\CC^2$ becomes left multiplication with unit quaternions, upon identifying $\CC^2=\HH$ as above.
\item We consider $\SO(4)\cong\SU(2)\times \SU(2)/\ZZ_2$, by letting $\SU(2)\times \SU(2)$ act on $\RR^4=\HH$ via quaternionic multiplication $(v,w)\cdot h = vhw^{-1}$. By (iii), the inclusion $\SU(2)\rightarrow \SO(4)$ onto the left hand factor agrees with the inclusion $\SU(2)\subset \U(2)\rightarrow \SO(4)$ from (i)\item We use the embedding $\SO(4)\rightarrow \SO(5)$ in the upper left block.
\end{enumerate}
\end{rem}

\subsection{Preliminaries to the construction}

We consider, for a fixed integer $n$, the map 
\[
f_{n} \colon \CC\to \CC;\, z\mapsto \begin{cases} z^n/|z|^{n-1} & z\neq 0 \\ 0 & z = 0.\end{cases}
\] 
On $\CC\setminus \{0\}$ it is a (real) smooth map and in $z=0$ it is continuous, since $|f_{n}(z)| = |z|$ for all $z$. Then we define, for integers $n,m$ a map
\[
A_{n,m}:S^3\longrightarrow \SU(2);\,
A_{n,m}(z_1,z_2):=\left(\begin{matrix}
		f_{n}(z_{1}) & -f_{-m}(z_{2}) \\ 
		f_{m}(z_{2}) & f_{-n}(z_{1})
\end{matrix}\right).
\]
From the description
\[
\SU(2)=\left\{\left. \left(\begin{matrix}z & -\overline{w}\\ w & \overline{z}\end{matrix}\right)\, \right|\, |z|^2+|w|^2=1\right\}
\]
one checks directly that $A_{n,m}$ indeed takes values in $\SU(2)$. Consider the standard $T^2$-action on $S^3$ 
\[
(s,t)\cdot (z,w) = (sz,tw).
\]
Via the identification $S^3\cong \SU(2); (z,w)\mapsto \left(\begin{matrix}z & -\overline{w} \\ w & \overline{z}\end{matrix}\right)$ this action becomes
\[
(s,t)\cdot \left(\begin{matrix}z & -\overline{w} \\ w & \overline{z}\end{matrix}\right)=\left(\begin{matrix}sz & -\overline{t}\overline{w} \\ tw &\overline{s} \overline{z}\end{matrix}\right).
\]

\begin{lem}\label{lem:Aequiv}
The map $A_{n,m}:S^3\to \SU(2)$ satisfies the following equivariance property: \[
A_{n,m}((s,t)\cdot (z,w)) = (s^{n},t^{m})\cdot A_{n,m}(z,w).
\]
\end{lem}
\begin{proof}
\begin{align*}
A_{n,m}(sz,tw) &= \left(\begin{matrix}
		f_{n}(sz) & -f_{-m}(tw) \\ 
		f_{m}(tw) & f_{-n}(sz)
\end{matrix}\right) \\
&= \left(\begin{matrix}
		s^{n}f_{n}(z) & -t^{-m}f_{-m}(w) \\ 
		t^{m}f_{m}(w) & s^{-n}f_{-n}(z)
\end{matrix}\right) = (s^{n},t^{m})\cdot A_{n,m}(z,w).
\end{align*}
\end{proof}

For reference in several computations below, we mention the following lemma:

\begin{lem}\label{lem:equivH}
For any $B\in \SU(2)$ and any $s,t\in S^1$ we have
\[
\left(\begin{matrix}\overline{s} & 0 \\ 0 & \overline{t}\end{matrix}\right)\cdot B \cdot \left(\begin{matrix} s & 0 \\ 0 & t\end{matrix}\right) = (1,s\overline{t})\cdot B.
\]
\end{lem}
\begin{proof} We write $B = \left(\begin{matrix}z & -\overline{w} \\ w & \overline{z} \end{matrix}\right)$ and compute
\[
\left(\begin{matrix}\overline{s} & 0 \\ 0 & \overline{t}\end{matrix}\right)\cdot
B \cdot \left(\begin{matrix}s & 0 \\ 0 & t\end{matrix}\right)
 = \left(\begin{matrix}\overline{s} & 0 \\ 0 & \overline{t}\end{matrix}\right)\cdot
\left(\begin{matrix}z & -\overline{w} \\ w & \overline{z} \end{matrix}\right) \cdot \left(\begin{matrix}s & 0 \\ 0 & t\end{matrix}\right)
 = \left(\begin{matrix}z & -\overline{s}t\overline{w} \\ s\overline{t} w & \overline{z}\end{matrix}\right) =  (1,s\overline{t})\cdot B.
\]
\end{proof}

In the following subsections, we will construct various $T^2$-equivariant $S^4$-bundles over the base $S^4$, equipped with the standard $T^2$-action, using explicit clutching maps involving the maps $A_{n,m}$. In every example, we consider two copies of $D^4\times S^4$, the first of which being endowed with the $T^2$-action
\begin{equation}\label{eq:firstaction}
(s,t)\cdot ((z,w),v) = ((sz,tw),\varphi(s,t)v),
\end{equation}
where 
\[
\varphi:T^2\longrightarrow \U(2)\subset \SO(4)\subset \SO(5);\, \varphi(s,t) = \left(\begin{matrix}s^a t^b & 0 \\ 0& s^ct^d \end{matrix}\right).
\]
In this way the associated GKM fibration will have the standard weights $(1,0)$ and $(0,1)$ in the base, and the weights $(a,b)$ and $(c,d)$ in one of the fibers, which is in alignment with our conventions in Section \ref{sec:combinatorial}. The task is to find gluing maps $\Psi:S^3\times S^4\to S^3\times S^4$ that are equivariant with respect to various $T^2$-actions on $D^4\times S^4$ of the form $(s,t)\cdot ((z,w),v) = ((sz,tw),\psi(s,t)v)$, with $\psi$ some homomorphism from $T^2$ into the standard maximal torus of $\U(2)\subset \SO(5)$, yielding the different types of orientable GKM graphs obtained in Section \ref{sec:combinatorial}.

Since the $A_{n,m}$ are merely continuous and not smooth, we close the preliminaries by discussing an equivariant smoothing procedure, which gives smooth structures on the examples such that the constructed actions are smooth. A sphere bundle constructed as above can be seen as the sphere bundle of a $T$-equivariant continuous real vector bundle $E\rightarrow S^4$ which arises as the respective gluing of two copies of $D^4\times \RR^5$.
From \cite[Remark 5.6]{MR4634089} we infer that there is a finite-dimensional linear
$T$-representation $W$ and a continuous equivariant map $f \colon S^4 \to
\mathrm{Gr}_{k}(W)$, such that $E\cong f^{\ast}(\gamma_{k})$ as $T$-equivariant vector bundles, where $\gamma_{k} \to \mathrm{Gr}_{k}(W)$ is the tautological bundle over the Grassmanian of linear subspaces of $W$ of
dimension $k$. 

The strategy is to equivariantly homotope $f$ to a smooth map, i.e. the homotopy should be equivariant w.r.t.\ the action on $S^4\times I$ via the first factor. This is indeed possible by

\begin{lem}[{\cite[Theorem 4.2]{MR0413144}}]\label{L: equivariantsmoothing}
Let $M$ and $N$ be smooth $T$-manifolds and $f \colon M \to N$ a continuous equivariant
map, then there is an equivariant homotopy to a smooth equivariant map $f_{1} \colon M \to
N$.
\end{lem}

If we apply this to $f$, we obtain a smooth equivariant vector bundle $E_{1} \to S^4$ which is
isomorphic to $E \to S^4$ (in the equivariant continuous category, see \cite[Theorem 8.15]{MR889050}).

Using a Riemannian metric on $E_1$ we get a smooth equivariant sphere bundle $S(E_1)$. An equivariant isomorphism $\tilde f\colon E\rightarrow E_1$ yields a homeomorphism $S(E)\rightarrow S(E_1)$, $e\mapsto\frac{\tilde{f} (e)}{\|\tilde{f} (e)\|}$. Note that the scalars $\frac{1}{\|\tilde{f}(e)\|}$ are constant along the $T$-orbits. Since the actions on $E,E_1$ are linear, it follows, that the map $S(E)\rightarrow S(E_1)$ is an equivariant homeomorphism. In this way, all sphere bundles constructed below can be equipped with a smooth structure such that the $T$-action is smooth.

\subsection{The case {\textbf{PA}$++$}}\label{subsec:PA++}

Consider on $S^4\times S^4$ the product $T^2$-action with weights $(1,0)$, $(0,1)$ in the first factor, and $(a,b)$, $(c,d)$ in the second. The projection onto the first factor defines a bundle which we denote
\[
M^{{\mathbf{PA}++}}_{a,b,c,d}\longrightarrow S^4
\]
whose GKM fibration is given by the following product type graph:
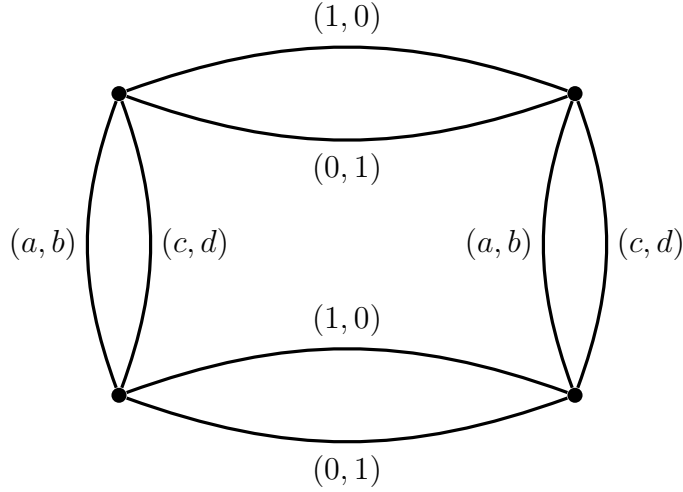
\begin{center}
		\begin{tikzpicture}

			\node (a) at (0,0)[circle,fill,inner sep=2pt] {};
			\node (b) at (6,0)[circle,fill,inner sep=2pt]{};
			\node (c) at (0,4)[circle,fill,inner sep=2pt]{};	
			\node (d) at (6,4)[circle,fill,inner sep=2pt]{};			
			
			\node at (-1,2) {$(a,b)$};
			\node at (+1,2) {$(c,d)$};
			\node at (5,2) {$(a,b)$};
			\node at (7,2) {$(c,d)$};
			\node at (3,5) {$(1,0)$};
			\node at (3,3) {$(0,1)$};
			\node at (3,1) {$(1,0)$};
			\node at (3,-1) {$(0,1)$};

			\draw [very thick](a) to[in=160, out=20] (b);
			\draw [very thick](a) to[in=200, out=-20] (b);
			\draw [very thick](c) to[in=160, out=20] (d);
			\draw [very thick](c) to[in=200, out=-20] (d);
			\draw [very thick](a) to[in=250, out=110] (c);
			\draw [very thick](a) to[in=290, out=70] (c);
			\draw [very thick](b) to[in=250, out=110] (d);
			\draw [very thick](b) to[in=290, out=70] (d);
			
		\end{tikzpicture}
	\end{center}

\subsection{The case {\textbf{PD}$++$}}\label{subsec:PD++}

We consider the bundle 
\[
M^{{\mathbf{PD}++}}_{a,b,c,d}\longrightarrow S^4
\]
given by the clutching function
\[
\Psi^{{\mathbf{PD}++}}_{a,b,c,d}:S^3\times S^4\longrightarrow S^3\times S^4;\, ((z,w),v)\longmapsto ((z,w),A_{c-a,d-b}(z,w) v),
\]
with $A_{c-a,d-b}(z,w)\in \SU(2)\subset \U(2)\subset \SO(4)\subset \SO(5)$ as in Remark \ref{rem: conventions}. On the second copy of $D^4\times S^4$ we consider the $T^2$-action
\begin{equation}\label{eq:secondaction1}
(s,t)\cdot ((z,w),v) = ((sz,tw),\psi(s,t) v).
\end{equation}
where $\psi(s,t):T^2\to \U(2)\subset \SO(4)\subset \SO(5)$ is defined by
\[
\psi(s,t) = \left(\begin{matrix}s^{c}t^b & 0 \\ 0 & s^{a}t^d\end{matrix}\right).
\]

\begin{lem}\label{lem:equiv1}
$\Psi^{{\mathbf{PD}++}}_{a,b,c,d}$ is equivariant with respect to the actions \eqref{eq:firstaction} and \eqref{eq:secondaction1} on $S^3\times S^4$.
\end{lem}
\begin{proof}
We compute
\begin{align*}
	\Psi^{{\mathbf{PD}++}}_{a,b,c,d}&((s,t)\cdot ((z,w),v)) \\
	&= ((sz,tw),A_{c-a,d-b}(sz,tw)\left(\begin{matrix}s^a t^b & 0 \\ 0 & s^ct^d  \end{matrix} \right)\cdot v) \\
&= ((sz,tw),[(s^{c-a},t^{d-b})\cdot A_{c-a,d-b}(z,w)] \left(\begin{matrix}s^a t^b & 0 \\ 0 & s^ct^d  \end{matrix} \right) \cdot v)\\
&= ((sz,tw),\left(\begin{matrix}s^a t^b & 0 \\ 0 & s^ct^d  \end{matrix} \right) [(1,s^{a-c}t^{b-d})\cdot (s^{c-a},t^{d-b})\cdot A_{c-a,d-b}(z,w)]\cdot v)\\
&= ((sz,tw),\left(\begin{matrix}s^a t^b & 0 \\ 0 & s^ct^d  \end{matrix} \right) [(s^{c-a},s^{a-c})\cdot A_{c-a,d-b}(z,w)]\cdot v\\
&= ((sz,tw),\left(\begin{matrix}s^a t^b & 0 \\ 0 & s^ct^d  \end{matrix} \right) \left(\begin{matrix}s^{c-a} & 0 \\ 0 & s^{a-c}  \end{matrix} \right)  A_{c-a,d-b}(z,w)\cdot v)\\
&=((sz,tw),\psi(s,t)A_{c-a,d-b}(z,w)\cdot v\\
&=(s,t)\cdot \Psi^{{\mathbf{PD}++}}_{a,b,c,d}((z,w),v)
\end{align*}
where we used Lemma \ref{lem:Aequiv} for the second, and Lemma  \ref{lem:equivH} for the third equality.
\end{proof}

Both copies of $D^4\times S^4$ contain $D^4\times \{(0,0,0,0,\pm 1)\}$ as $T^2$-invariant subspaces, which glue to two $T^2$-invariant copies of $S^4$. Hence the GKM fibration is of product type:
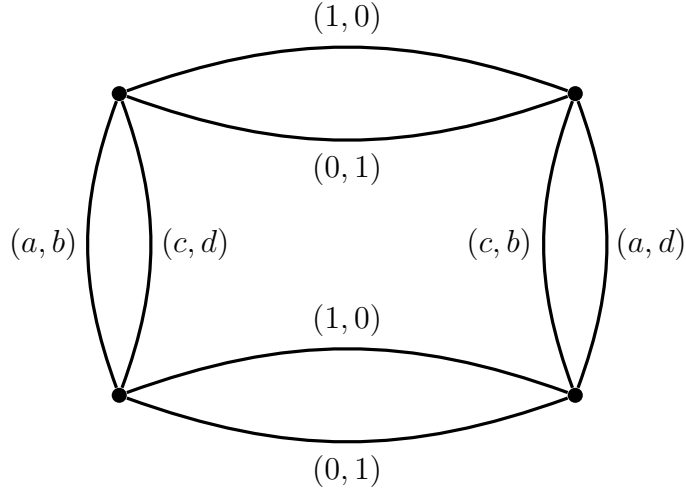
\begin{center}
		\begin{tikzpicture}

			\node (a) at (0,0)[circle,fill,inner sep=2pt] {};
			\node (b) at (6,0)[circle,fill,inner sep=2pt]{};
			\node (c) at (0,4)[circle,fill,inner sep=2pt]{};	
			\node (d) at (6,4)[circle,fill,inner sep=2pt]{};			
			
			\node at (-1,2) {$(a,b)$};
			\node at (+1,2) {$(c,d)$};
			\node at (5,2) {$(c,b)$};
			\node at (7,2) {$(a,d)$};
			\node at (3,5) {$(1,0)$};
			\node at (3,3) {$(0,1)$};
			\node at (3,1) {$(1,0)$};
			\node at (3,-1) {$(0,1)$};

			\draw [very thick](a) to[in=160, out=20] (b);
			\draw [very thick](a) to[in=200, out=-20] (b);
			\draw [very thick](c) to[in=160, out=20] (d);
			\draw [very thick](c) to[in=200, out=-20] (d);
			\draw [very thick](a) to[in=250, out=110] (c);
			\draw [very thick](a) to[in=290, out=70] (c);
			\draw [very thick](b) to[in=250, out=110] (d);
			\draw [very thick](b) to[in=290, out=70] (d);
			
		\end{tikzpicture}
	\end{center}

\subsection{The case {\textbf{PD}$--$}}
Although, as explained in Section \ref{sec:combinatorial}, the cases ${\mathbf{D}}++$ and ${\mathbf{D}}--$ are not to be distinguished, we describe another class of gluing maps that give the same graphs as those in the previous section. To distinguish them, we use superscripts ${\mathbf{PD}}--$. They will be used in the subsequent case for the definition of yet another class of gluing maps.

In this case the gluing map takes values in the second $\SU(2)$ factor of $\SU(2)\times \SU(2)/\ZZ_2 = \SO(4)\subset \SO(5)$. We consider the bundle
\[
M^{{\mathbf{PD}}--}_{a,b,c,d}\longrightarrow S^4
\]
defined by the gluing map
\[
\Psi^{{\mathbf{PD}}--}_{a,b,c,d}:S^3\times S^4\longrightarrow S^3\times S^4;\, ((z,w),v)\longmapsto ((z,w),({\mathrm{I}_2},A_{a+c,b+d}(z,w))\cdot v).
\]
On the second copy of $D^4\times S^3$ we consider the $T^2$-action 
\begin{equation}\label{eq:secondaction2}
(s,t)\cdot ((z,w),v) = ((sz,tw),\psi(s,t) v).
\end{equation}
where $\psi(s,t):T^2\to \U(2)\subset \SO(4)\subset \SO(5)$ is defined by
\[
\psi(s,t) = \left(\begin{matrix}s^{-c}t^b & 0 \\ 0 & s^{-a}t^d\end{matrix}\right).
\]

\begin{lem}\label{lem:equiv2}
$\Psi^{{\mathbf{PD}--}}_{a,b,c,d}$ is equivariant with respect to the actions \eqref{eq:firstaction} and \eqref{eq:secondaction2} on $S^3\times S^4$.
\end{lem}
\begin{proof}
We write an element of $S^4$ as $v = (q,u)$, for $q=q_1+jq_2\in \HH$ and $u\in \RR$. Then we compute
\begin{align*}
\Psi^{{\mathbf{PD}--}}_{a,b,c,d}((s,t)\cdot ((z,w),(q,u)) &= \Psi^{{\mathbf{PD}--}}_{a,b,c,d}((sz,tw),\varphi(s,t) (q,u))\\
&=((sz,tw),(({\mathrm{I}_2},A_{a+c,b+d}(sz,tw))(s^at^bq_1+js^ct^dq_2),u)\\
&=((sz,tw),((s^at^bq_1 + js^ct^dq_2)(\overline{f_{a+c}(sz)} - jf_{b+d}(tw)),u))\\
&=((sz,tw),((s^at^bq_1 + js^ct^dq_2)(s^{-a-c}\overline{f_{a+c}(z)} - jt^{b+d}f_{b+d}(w)),u))\\
&=((sz,tw),((s^{-c}t^bq_1\overline{f_{a+c}(z)}+s^{-c}t^b\overline{q_2}f_{b+d}(w) \\
&\qquad\qquad \qquad\qquad + j(s^{-a}t^d q_2\overline{f_{a+c}(z)} - s^{-a}t^d\overline{q_1} f_{b+d}(w)),u))\\
&= ((sz,tw),\psi(s,t)(({\mathrm{I}_2},A_{a+c,b+d}(z,w))q,u))\\
&=(s,t)\cdot \Psi^{{\mathrm{PD}}--}_{a,b,c,d}((z,w),(q,u)).
\end{align*}
\end{proof}
For the same reason as in the previous case, the GKM fibration is of product type:

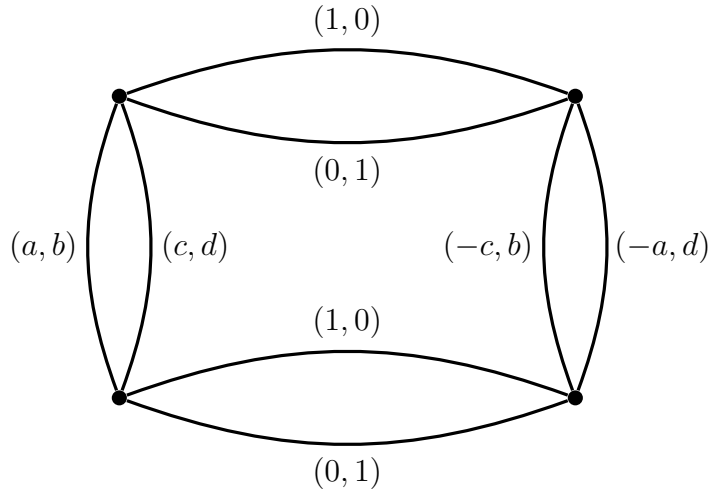
\begin{center}
		\begin{tikzpicture}

			\node (a) at (0,0)[circle,fill,inner sep=2pt] {};
			\node (b) at (6,0)[circle,fill,inner sep=2pt]{};
			\node (c) at (0,4)[circle,fill,inner sep=2pt]{};	
			\node (d) at (6,4)[circle,fill,inner sep=2pt]{};			
			
			\node at (-1,2) {$(a,b)$};
			\node at (+1,2) {$(c,d)$};
			\node at (4.85,2) {$(-c,b)$};
			\node at (7.15,2) {$(-a,d)$};
			\node at (3,5) {$(1,0)$};
			\node at (3,3) {$(0,1)$};
			\node at (3,1) {$(1,0)$};
			\node at (3,-1) {$(0,1)$};

			\draw [very thick](a) to[in=160, out=20] (b);
			\draw [very thick](a) to[in=200, out=-20] (b);
			\draw [very thick](c) to[in=160, out=20] (d);
			\draw [very thick](c) to[in=200, out=-20] (d);
			\draw [very thick](a) to[in=250, out=110] (c);
			\draw [very thick](a) to[in=290, out=70] (c);
			\draw [very thick](b) to[in=250, out=110] (d);
			\draw [very thick](b) to[in=290, out=70] (d);
			
		\end{tikzpicture}
	\end{center}
As observed before, this (unsigned) graph is identical to that of $M^{\mathbf{PD}++}_{a,b,-c,-d}$ from the previous section.

\subsection{The case {\textbf{PA$--$}}}

We construct a bundle
\[
M^{{\mathbf{PA}}--}_{a,b,c,d}\longrightarrow S^4
\]
defined by the gluing map
\[
\Psi^{{\mathbf{PA}}--}_{a,b,c,d} = \Psi^{{\mathbf{PD}}--}_{c,b,a,d}\circ \Psi^{{\mathbf{PD}}++}_{a,b,c,d}:S^3\times S^4\longrightarrow S^3\times S^4.
\]
On the second copy of $D^4\times S^3$ we consider the $T^2$-action 
\begin{equation}\label{eq:secondaction3}
(s,t)\cdot ((z,w),v) = ((sz,tw),\psi(s,t) v).
\end{equation}
where $\psi(s,t):T^2\to \U(2)\subset \SO(4)\subset \SO(5)$ is defined by
\[
\psi(s,t) = \left(\begin{matrix}s^{-a}t^b & 0 \\ 0 & s^{-c}t^d\end{matrix}\right).
\]

\begin{lem}
$\Psi^{{\mathbf{PA}--}}_{a,b,c,d}$ is equivariant with respect to the actions \eqref{eq:firstaction} and \eqref{eq:secondaction3} on $S^3\times S^4$.
\end{lem}
\begin{proof}
This follows immediately from Lemmas \ref{lem:equiv1} and \ref{lem:equiv2}.
\end{proof}
Again, the GKM fibration is of product type:

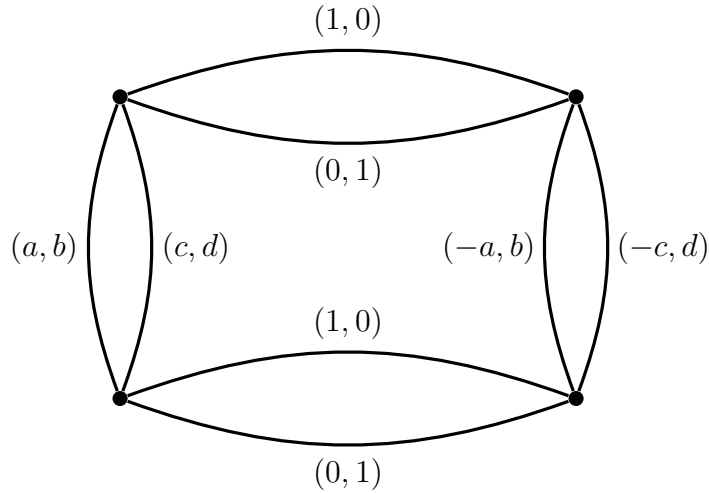
\begin{center}
		\begin{tikzpicture}

			\node (a) at (0,0)[circle,fill,inner sep=2pt] {};
			\node (b) at (6,0)[circle,fill,inner sep=2pt]{};
			\node (c) at (0,4)[circle,fill,inner sep=2pt]{};	
			\node (d) at (6,4)[circle,fill,inner sep=2pt]{};			
			
			\node at (-1,2) {$(a,b)$};
			\node at (+1,2) {$(c,d)$};
			\node at (4.85,2) {$(-a,b)$};
			\node at (7.15,2) {$(-c,d)$};
			\node at (3,5) {$(1,0)$};
			\node at (3,3) {$(0,1)$};
			\node at (3,1) {$(1,0)$};
			\node at (3,-1) {$(0,1)$};

			\draw [very thick](a) to[in=160, out=20] (b);
			\draw [very thick](a) to[in=200, out=-20] (b);
			\draw [very thick](c) to[in=160, out=20] (d);
			\draw [very thick](c) to[in=200, out=-20] (d);
			\draw [very thick](a) to[in=250, out=110] (c);
			\draw [very thick](a) to[in=290, out=70] (c);
			\draw [very thick](b) to[in=250, out=110] (d);
			\draw [very thick](b) to[in=290, out=70] (d);
			
		\end{tikzpicture}
	\end{center}

\subsection{The case {\textbf{TA}$+-$}}
For this case, we consider $\rho:\SU(2)\to \SO(3)$ the standard $2:1$ covering homomorphism, such that the standard maximal torus (the diagonal circle) is sent to the $\U(1)=\SO(2)$ embedded in the upper left, via \[\rho\left(\begin{matrix}e^{i\alpha} & 0 \\ 0 & e^{-i\alpha}\end{matrix}\right) = \left(\begin{matrix}\cos (2\alpha) & -\sin (2\alpha) &0\\ \sin(2\alpha) & \cos(2\alpha)& 0\\ 0& 0& 1\end{matrix}\right).\] We consider $\rho:\SU(2)\to \SO(3)\subset \SO(5)$, with the image embedded in the lower right corner.

\begin{rem}\label{rem:twistexplicit}
For later purpose we give an explicit description of $\rho$. We write \[\left(\begin{matrix}z & -\overline{w}\\ w & \overline{z} \end{matrix}\right)=z+jw\] as before and identify $\RR^3$ with the imaginary quaternions using the basis $j,k,i$ (in this order). Then the maximal torus corresponds to elements $z+j\cdot 0$ and $\rho$ is given by letting a unit quaternion $q$ act on an imaginary quaternion $h$ as $qhq^{-1}$. We note that $\rho(z+j\cdot 0)$, when considered as an element of $\SO(5)$ as described above, has $1$ as the bottom right entry while the same entry of $\rho(0+jz)$ is $-1$.
\end{rem}

Now we construct a bundle
\[
M^{{\textbf{TA}}+-}_{a,b,c,d}\longrightarrow S^4
\]
defined by the gluing map
\[
\Psi^{{\textbf{TA}}+-}_{a,b,c,d}:S^3\times S^4\longrightarrow S^3\times S^4;\, ((z,w),v) \longmapsto ((z,w),\rho(A_{-c,-d}(z,w))v).
\]
On the second copy of $D^4\times S^3$ we consider the $T^2$-action 
\begin{equation}\label{eq:secondaction4}
(s,t)\cdot ((z,w),v) = ((sz,tw),\psi(s,t) v).
\end{equation}
where $\psi(s,t):T^2\to \U(2)\subset \SO(4)\subset \SO(5)$ is defined by
\[
\psi(s,t) = \left(\begin{matrix}s^{a}t^b & 0 \\ 0 & s^{-c}t^d\end{matrix}\right).
\]
\begin{lem}\label{lem:equiv4}
$\Psi^{{\mathbf{TA}+-}}_{a,b,c,d}$ is equivariant with respect to the actions \eqref{eq:firstaction} and \eqref{eq:secondaction4} on $S^3\times S^4$.
\end{lem}
\begin{proof} In the ensuing computation we write matrices in $\SO(5)$ as $\U(2)\times \{1\}$ block matrices (see Remark \ref{rem: conventions} (i)) whenever possible. Note that while the image of $\rho$ is not in general of this form, it is on the maximal torus of $\SU(2)$. In $\SO(5)$ we compute
\begin{align*}
(\rho&\circ  A_{-c,-d})(sz,tw)\cdot \varphi(s,t)\\
&= \rho((s^{-c},t^{-d})\cdot A_{-c,-d}(z,w))\cdot \left(\begin{matrix}
s^at^b & 0 & 0 \\ 0 & s^ct^d & 0 \\ 0 & 0 & 1
\end{matrix} \right)\\
&=\left(\begin{matrix}
s^at^b & 0 & 0 \\ 0 & s^ct^d & 0 \\ 0 & 0 & 1
\end{matrix} \right)\cdot \left(\begin{matrix}
1 & 0 & 0 \\ 0 & s^{-c}t^{-d} & 0 \\ 0 & 0 & 1
\end{matrix} \right)\rho([(s^{-c},t^{-d})\cdot A_{-c,-d}(z,w)])\cdot \left(\begin{matrix}
1 & 0 & 0 \\ 0 & s^ct^d & 0 \\ 0 & 0 & 1
\end{matrix} \right)\\
&= \left(\begin{matrix}
s^at^b & 0 & 0 \\ 0 & s^ct^d & 0 \\ 0 & 0 & 1
\end{matrix} \right)\rho\left(\left(\begin{matrix}s^{-\frac{c}{2}}t^{-\frac{d}{2}} & 0 \\ 0 & s^{\frac{c}{2}}t^{\frac{d}{2}}\end{matrix}\right)\cdot [(s^{-c},t^{-d})\cdot A_{-c,-d}(z,w)]\cdot \left(\begin{matrix}s^{\frac{c}{2}}t^{\frac{d}{2}} & 0 \\ 0 & s^{-\frac{c}{2}}t^{-\frac{d}{2}}\end{matrix}\right)\right)\cdot v\\
&=\left(\begin{matrix}
s^at^b & 0 & 0 \\ 0 & s^ct^d & 0 \\ 0 & 0 & 1
\end{matrix} \right)\rho([(s^{-c},s^ct^dt^{-d})\cdot A_{-c,-d}(z,w)]) \\
&= \left(\begin{matrix}
s^at^b & 0 & 0 \\ 0 & s^ct^d & 0 \\ 0 & 0 & 1
\end{matrix} \right)\cdot\rho \left(\left(\begin{matrix}s^{-c} & 0 \\ 0 & s^{c}\end{matrix}\right)\right)\cdot \rho(A_{-c,-d}(z,w))\\
&= \left(\begin{matrix}
s^at^b & 0 & 0 \\ 0 & s^ct^d & 0 \\ 0 & 0 & 1
\end{matrix} \right)\cdot \left(\begin{matrix}
1 & 0 & 0 \\ 0 & s^{-2c} & 0 \\ 0 & 0 & 1
\end{matrix} \right)\cdot \rho(A_{-c,-d}(z,w))\\
&=\left(\begin{matrix}
s^at^b & 0 & 0 \\ 0 & s^{-c}t^d & 0 \\ 0 & 0 & 1
\end{matrix} \right) \cdot \rho(A_{-c,-d}(z,w))
\end{align*}
where we used Lemma \ref{lem:equivH} for the fourth equality, and hence
\begin{align*}
\Psi^{{\mathbf{TA}--}}_{a,b,c,d}((s,t)&\cdot ((z,w),v)) = \Psi^{{\mathbf{TA}--}}_{a,b,c,d}((sz,tw),\varphi(s,t)\cdot v)\\
&= ((sz,tw),(\rho\circ  A_{-c,-d})(sz,tw)\cdot \varphi(s,t)\cdot v)\\
&=((sz,tw),\psi(s,t)\cdot \rho(A_{-c,-d}(z,w))\cdot v) = (s,t)\cdot \Psi^{{\mathbf{TA}--}}_{a,b,c,d}((z,w), v).
\end{align*}
\end{proof}

The GKM graph of $M^{{\mathbf{TA}}+-}_{a,b,c,d}$ is of twisted type, because by Proposition \ref{prop:combinatorialclassification} a graph of type {\textbf{PA$+-$}} is not orientable and hence impossible to appear as the GKM graph of a GKM manifold. To see this explicitly we observe that the two invariant spheres in the base sphere intersect the equator $S^3\cong \SU(2)$ in two circles $C_1$ and $C_2$ given by unit quaternions of the form $z+j\cdot 0$ and $0+j\cdot z$ for $z\in S^1$. Then by Remark \ref{rem:twistexplicit} we see that the points $(0,\ldots,0,\pm 1)$ of the fiber $S^4$ get swapped under the gluing map over points in $C_2$ but are fixed over points in $C_1$. Consequently, when following the two invariant horizontal spheres emanating from a fixed point of the total space one arrives at two distinct fixed points.
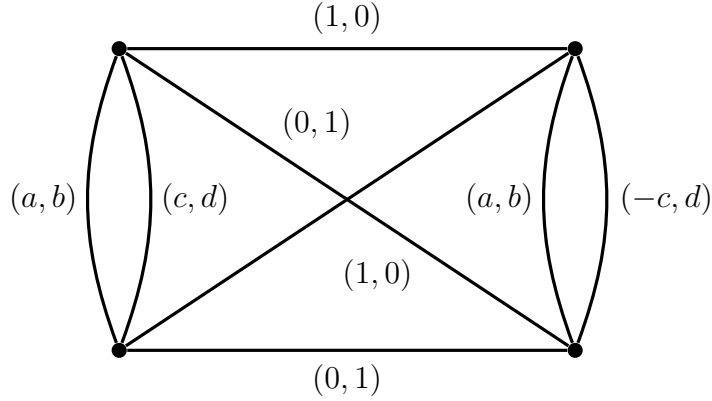
\begin{center}
		\begin{tikzpicture}

			\node (a) at (0,0)[circle,fill,inner sep=2pt] {};
			\node (b) at (6,0)[circle,fill,inner sep=2pt]{};
			\node (c) at (0,4)[circle,fill,inner sep=2pt]{};	
			\node (d) at (6,4)[circle,fill,inner sep=2pt]{};			
			
			\node at (-1,2) {$(a,b)$};
			\node at (+1,2) {$(c,d)$};
			\node at (5,2) {$(a,b)$};
			\node at (7.2,2) {$(-c,d)$};
			\node at (3,4.4) {$(1,0)$};
			\node at (2.6,3) {$(0,1)$};
			\node at (3.4,1) {$(1,0)$};
			\node at (3,-.4) {$(0,1)$};

			\draw [very thick](a) to (b);
			\draw [very thick](a) to (d);
			\draw [very thick](c) to (d);
			\draw [very thick](c) to (b);
			\draw [very thick](a) to[in=250, out=110] (c);
			\draw [very thick](a) to[in=290, out=70] (c);
			\draw [very thick](b) to[in=250, out=110] (d);
			\draw [very thick](b) to[in=290, out=70] (d);
			
		\end{tikzpicture}
	\end{center}

\subsection{The case {\textbf{TD}$+-$}}
Finally, we construct a bundle
\[
M^{{\mathbf{TD}}+-}_{a,b,c,d}\longrightarrow S^4
\]
via the gluing map
\[
\Psi^{{\mathbf{TD}}+-}_{a,b,c,d} = \Psi^{{\mathbf{PD}}++}_{a,b,-c,d}  \circ \Psi^{{\mathbf{TA}}+-}_{a,b,c,d}:S^3\times S^4\longrightarrow S^3\times S^4.
\]
On the second copy of $D^4\times S^3$ we consider the $T^2$-action 
\begin{equation}\label{eq:secondaction5}
(s,t)\cdot ((z,w),v) = ((sz,tw),\psi(s,t) v).
\end{equation}
where $\psi(s,t):T^2\to \U(2)\subset \SO(4)\subset \SO(5)$ is defined by
\[
\psi(s,t) = \left(\begin{matrix}s^{-c}t^b & 0 \\ 0 & s^{a}t^d\end{matrix}\right).
\]

\begin{lem}
$\Psi^{{\mathbf{TD}+-}}_{a,b,c,d}$ is equivariant with respect to the actions \eqref{eq:firstaction} and \eqref{eq:secondaction5} on $S^3\times S^4$.
\end{lem}
\begin{proof}
This follows immediately from Lemmas \ref{lem:equiv1} and \ref{lem:equiv4}.
\end{proof}

As in the previous case, the GKM graph is necessarily of twisted type.
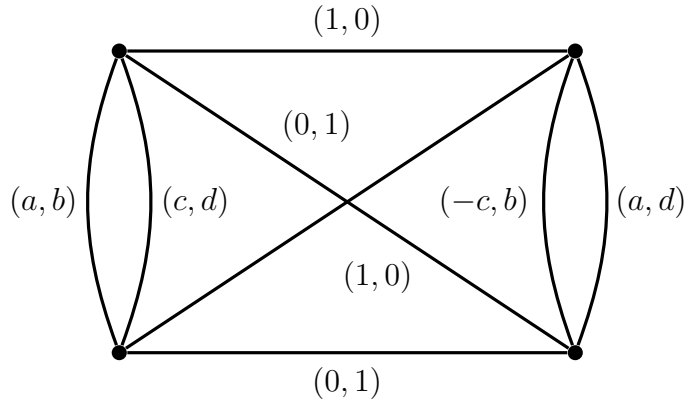
\begin{center}
		\begin{tikzpicture}

			\node (a) at (0,0)[circle,fill,inner sep=2pt] {};
			\node (b) at (6,0)[circle,fill,inner sep=2pt]{};
			\node (c) at (0,4)[circle,fill,inner sep=2pt]{};	
			\node (d) at (6,4)[circle,fill,inner sep=2pt]{};			
			
			\node at (-1,2) {$(a,b)$};
			\node at (+1,2) {$(c,d)$};
			\node at (7,2) {$(a,d)$};
			\node at (4.8,2) {$(-c,b)$};
			\node at (3,4.4) {$(1,0)$};
			\node at (2.6,3) {$(0,1)$};
			\node at (3.4,1) {$(1,0)$};
			\node at (3,-.4) {$(0,1)$};

			\draw [very thick](a) to (b);
			\draw [very thick](a) to (d);
			\draw [very thick](c) to (d);
			\draw [very thick](c) to (b);
			\draw [very thick](a) to[in=250, out=110] (c);
			\draw [very thick](a) to[in=290, out=70] (c);
			\draw [very thick](b) to[in=250, out=110] (d);
			\draw [very thick](b) to[in=290, out=70] (d);
			
		\end{tikzpicture}
	\end{center}

\section{Topological classification} \label{sec:topclassification}

As described in Section \ref{SubSec: Sphere bundles}, any element $\varphi\in \pi_3(\SO(5))$ defines, by clutching, a smooth $S^4$-bundle $M_\varphi\to S^4$ over $S^4$. The nonequivariant homotopy types, homeomorphism types, and diffeomorphism types of the $M_\varphi$ are well-understood in terms of the clutching element $\varphi$.  We may also speak about the clutching \emph{number} of such $S^4$-bundles, using that $\pi_3(\SO(5))\cong \ZZ$. In the following, we do not distinguish any of the two generators of $\pi_3(\SO(5))$, hence our clutching numbers are well-defined only up to sign. It was shown by James--Whitehead \cite[p.\ 217]{JamesWhitehead} that the total space $M_\varphi$ is homotopy equivalent to $M_\psi$ if and only if
\[\varphi=\pm \psi\mod 24,\]
hence leaving us with $13$ distinct homotopy types. Furthermore one has

\begin{lem}\label{lem:homeoclutching}
The total spaces of the $S^4$-bundles $M_\varphi$ and $M_\psi$ are homeomorphic if and only if $\varphi=\pm \psi$. In this case they are also diffeomorphic.
\end{lem}
\begin{proof}
If $\psi=-\varphi$, consider the diffeomorphism $g:S^4\to S^4$ which, on each hemisphere $D^4$, is given by $(z_1,z_2)\mapsto (z_1,\overline{z_2})$. Then the pullback $g^*M_\varphi$ has clutching function $a\circ g|_{S^3}=-\varphi$. Hence $M_\varphi$ and $ M_{-\varphi}$ are diffeomorphic.

Conversely any homeomorphism $h\colon M_\varphi\rightarrow M_\psi$ satisfies $h^*(p_1(M_\psi))=p_1(M_\varphi)$ on the rational Pontryagin classes by \cite{Novikov}. Since the integral cohomology is torsion free the identity holds for integral Pontryagin classes as well. Then $\varphi=\pm \psi$ follows from Corollary \ref{C: clutching function pontryagin}.
\end{proof}

\begin{thm} \label{thm: clutching numbers}
The clutching numbers (up to sign) of the $T^2$-equivariant $S^4$-bundles over $S^4$ constructed in Section \ref{sec:construction} are given in the following table.
\begin{center}
\begin{tabular}{|c|c|}
\hline
& clutching number \\
\hline
$M^{{\mathbf{PA}}++}_{a,b,c,d}$ & $0$ \\
\hline
$M^{{\mathbf{PD}}++}_{a,b,c,d}$ & $\pm (a-c)(b-d)$ \\
\hline
$M^{{\mathbf{PA}}--}_{a,b,c,d}$ & $\pm 2(ad+bc)$\\
\hline
$M^{{\mathbf{TA}}+-}_{a,b,c,d}$ & $\pm 2cd$ \\
\hline
$M^{{\mathbf{TD}}+-}_{a,b,c,d}$ & $\pm (a(b-d) + c(b+d))$ \\
\hline
\end{tabular}
\end{center}
In particular, two of the corresponding total spaces are non-equivariantly homeomorphic (and diffeomorphic) if and only if the above numbers agree up to sign. They are homotopy equivalent if and only if the numbers are equal modulo $24$ and sign.
\end{thm}
\begin{proof}
By Corollary \ref{C: clutching function pontryagin} we only need to compute the first Pontryagin class in each case and argue that it is twice the number given in the theorem times the image of a generator coming from the base $S^4$. To this end, we calculate the first equivariant Pontryagin class of each example as described in Proposition \ref{P: Computation PC}.

In the following, by the left hand fiber (referring to the graphs as drawn in Section \ref{sec:construction}) we mean the $T$-invariant fiber with edge labels $(a,b)$, $(c,d)$. By the right hand fiber we mean the other $T$-invariant fiber.  In each of the cases, in the fixed points belonging to the left hand fiber, the first equivariant Pontryagin class restricts to
\[
x^2+y^2+(ax+by)^2+(cx+dy)^2 = (a^2+c^2+1)x^2+(b^2+d^2+1)y^2+2(ab+cd)xy.
\]
We carry out the computation for $M^{{\mathbf{PD}++}}_{a,b,c,d}=: M$. In the right hand fiber, the value at both fixed points is 
\[
x^2+y^2+(cx+by)^2+(ax+dy)^2 = (a^2+c^2+1)x^2+(b^2+d^2+1)y^2+2(bc+ad)xy.
\]
Modulo $R^4\cdot 1$ this is equivalent to the element which is $0$ in the left hand fiber and $-2(a-c)(b-d)xy$ in the right hand fiber (by subtracting the constant class with the value of the left hand fiber). Consider the commutative diagram

\[\xymatrix{
H_T^*(M^T)  & \ar[l]\ar[r] H_T^*(M) & H^*(M)\\
H_T^*((S^4)^T)\ar[u]  & \ar[l]\ar[r] \ar[u] H_T^*(S^4) & H^*(S^4)\ar[u]
}\]
for which we have just computed the image of $p_1^T(M)$ under the upper left map. The element in $H^4((S^4)^T)$ which is $0$ on the fixed point under the left hand fiber and $xy$ in the other one comes from $H_T^4(S^4)$ and restricts to a generator of $H^4(S^4)$. Its image $\alpha$ in $H_T^*(M^T)$ is the element which is zero in the two fixed points in the left hand fiber, and $xy$ in the right hand fiber. As argued above, modulo $R^4\cdot 1$, the previously computed image of $p_1^T(M)$ in $H_T(M^T)$ agrees with $-2(a-c)(b-d)\alpha$. It follows that the image $p_1(M)\in H^4(M)$ of $p_1^T(M)$ is $-2(a-c)(b-d)$ times the image of a generator coming from $H^4(S^4)$. This yields the second value in the table.

The argument in all other cases is identical; we only list the values of the first equivariant Pontryagin class at the fixed points in the right hand fiber. The respective values in the table then arise via substracting the previous value in the left hand fiber and dividing by $2$. They are:

For $M^{{\mathbf{PD}--}}_{a,b,c,d}$:
\[
x^2+y^2+(-cx+by)^2+(-ax+dy)^2 = (a^2+c^2+1)x^2+(b^2+d^2+1)y^2-2(bc+ad)xy.
\]
For $M^{{\mathbf{PA}--}}_{a,b,c,d}$:
\[
x^2+y^2+(-ax+by)^2+(-cx+dy)^2 = (a^2+c^2+1)x^2+(b^2+d^2+1)y^2-2(ab+cd)xy.
\]
For $M^{{\mathbf{TA}+-}}_{a,b,c,d}$:
\[
x^2+y^2+(ax+by)^2+(-cx+dy)^2 = (a^2+c^2+1)x^2+(b^2+d^2+1)y^2+2(ab-cd)xy.
\]
Finally, for $M^{{\mathbf{TD}+-}}_{a,b,c,d}$:
\[
x^2+y^2+(ax+dy)^2+(-cx+by)^2 = (a^2+c^2+1)x^2 + (b^2+d^2+1)y^2 + 2(ad-bc)xy.
\]
\end{proof}

\begin{proof}[Proof of Theorem \ref{thm:examples}]
Recall for any choice of $a,b,c,d$ subject to the restrictions in Remark \ref{rem: labels explicit} we obtain the respective GKM manifolds from Theorem \ref{thm: clutching numbers}. In particular, we find many examples of
pairs of homotopy equivalent, not homeomorphic GKM manifolds, e.g., $M^{{\mathbf{PD}}++}_{1,-1,2,23}$ is homotopy equivalent but not homeomorphic to $S^4\times S^4$.
\end{proof}

\begin{rem} \label{rem:blowup} Our class of examples also contains many pairs of GKM actions on the same smooth manifold whose GKM graphs do not agree as unlabled graphs. For instance, $M^{{\mathbf{TD}}+-}_{2,1,1,3}$ is nonequivariantly diffeomorphic to $S^4\times S^4$ while its GKM fibration is of twisted type. 

Such examples were known before, even in the symplectic toric setting. For instance, consider the product $\CC P^2\times \CC P^1$ with the standard toric $T^3$-action. The toric symplectic manifold one obtains by two equivariant blow-ups in fixed points depends on the choice of the two fixed points (also its GKM graph depends on it), but topologically it is always diffeomorphic to $(\CC P^2\times \CC P^1)\# \CC P^3 \# \CC P^3$. 
\end{rem}

\begin{rem}As shown in Corollary \ref{T: cohomology rings of S_l}, the cohomology ring of the total space $M_\varphi$, for $\varphi\in \pi_3({\mathrm{SO}}(5))\cong \ZZ$, contains only very little information; it is isomorphic to either $\ZZ[x,y]/(x^2,y^2)$ or $\ZZ[x,y]/(x^2-x\smile y,y^2)$, depending on whether the clutching number $\varphi$ is even or odd. In particular, among our examples there are many pairs of GKM manifolds with isomorphic cohomology rings which are not homotopy equivalent. Such examples were known before; in \cite{GKMnonrigid} we even found examples of GKM manifolds that are not homotopy equivalent but whose GKM graphs coincide.
\end{rem}

\section{Extensions}
\label{sec:extensions}

In this section we study for which of the $T^2$-equivariant bundles $S^4\rightarrow E\rightarrow S^4$ studied in the previous sections we can extend the torus action to a higher dimension. I.e.\ we assume $T^2$ acts in standard fashion on the base $S^4$, the action on $E$ is GKM and there is a torus $T\supset T^2$ and an extension of the actions on $E$ and $S^4$ such that the bundle is $T$-equivariant. Given such an extension one may choose a connection as in Lemma \ref{lem: goodcon} for the GKM graph of the $T$-action on $E$. We note that the same connection is then compatible with the restricted $T^2$-action and signs in the congruences resulting from sign choices are preserved as well. Hence, analogous to the discussion below Lemma \ref{lem: goodcon} the extension with its chosen connection can be associated to exactly one of the cases ${\mathbf{PA}}++$, ${\mathbf{PD}}++$, ${\mathbf{PA}}--$, ${\mathbf{TA}}+-$, or ${\mathbf{TD}}+-$.

\begin{thm}\label{thm: extensions}
Let $E\rightarrow S^4$ be a $T$-equivariant bundle as above and assume $T$ acts effectively on $E$. Then its dimension is bounded by
\begin{center}
\begin{tabular}{c|c}
	case & $\dim T\leq$\\ \hline
	$\mathbf{PA}++$ &  $4$\\
	$\mathbf{PD}++$ &  $3$\\
	$\mathbf{PA}--$ &  $2$\\
	$\mathbf{TA}+-$  & $3$\\
	$\mathbf{TD}+-$  & $2$
	\end{tabular}
	\end{center}
Furthermore the action on the equivariant bundles $M_{a,b,c,d}^*\rightarrow S^4$ (for $*$ one of the above cases) extends to the respective maximal dimension.
\end{thm}

\begin{rem}
Recall from Remark \ref{rem: overlap} that some of the bundles discussed can be associated to two distinct cases by different choices of connection. In this case it can happen that a connection is compatible with the $T^2$-action but not compatible with the extended $T$-action. Hence the restriction might be associated to a case (via a choice of connection) to which the extension does not belong. Thus, when talking about maximal extensions of the $T^2$-equivariant bundles without fixing the connection, then only the higher upper bound among the possibly two associated cases applies. As an example, the product bundle $S^4\times S^4$ with fiber weights $(a,b),(a,d)$ admits a connection of type $\mathbf{PA}++$ and one of type $\mathbf{PD}++$ while obviously admitting an extension to a $T^4$-action.
\end{rem}

\begin{proof}
We will discuss extensions to $T^4$-actions, however allowing non-effective extensions and hence covering all cases. The inclusion of $T^2$ into $T^4$ induces a projection $\ZZ^4\rightarrow \ZZ^2$ on the weight lattices, mapping the base weights onto $(1,0),(0,1)$. The resulting short exact sequence splits. Therefore we may pull back the $T^4$-action by an automorphism of $T^4$ such that the base weights are given by $(1,0,0,0)$ and $(0,1,0,0)$.

Using the congruence relation of the horizontal $(1,0,0,0)$-edges we may assume that the GKM graph of the extension is of the form

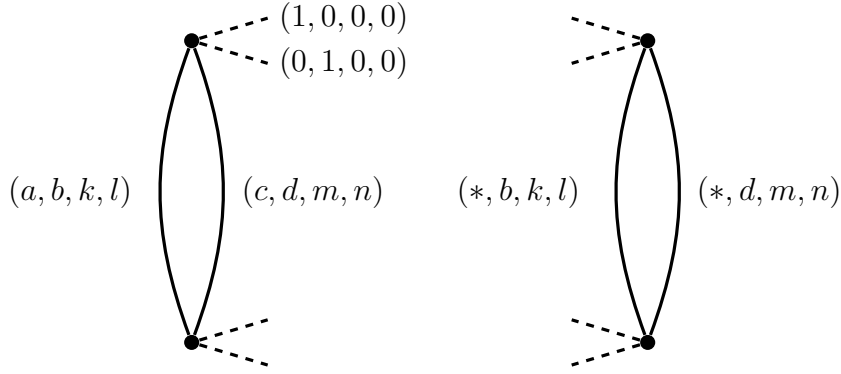
\begin{center}
		\begin{tikzpicture}

			\node (a) at (0,0)[circle,fill,inner sep=2pt] {};
			\node (b) at (6,0)[circle,fill,inner sep=2pt]{};
			\node (c) at (0,4)[circle,fill,inner sep=2pt]{};	
			\node (d) at (6,4)[circle,fill,inner sep=2pt]{};			
			
			\node at (-1.6,2) {$(a,b,k,l)$};
			\node at (+1.6,2) {$(c,d,m,n)$};
			\node at (7.6,2) {$(*,d,m,n)$};
			\node at (4.3,2) {$(*,b,k,l)$};
		
			\node at (2,4.3) {$(1,0,0,0)$};
			\node at (2,3.7) {$(0,1,0,0)$};
			
			\draw [very thick, dashed](a) --++(1,-0.3);
			\draw [very thick, dashed](a) --++(1,+0.3);
			
			\draw [very thick, dashed](c) --++(1,-0.3);
			\draw [very thick, dashed](c) --++(1,+0.3);
			
			\draw [very thick, dashed](b) --++(-1,-0.3);
			\draw [very thick, dashed](b) --++(-1,+0.3);
			
			\draw [very thick, dashed](d) --++(-1,-0.3);
			\draw [very thick, dashed](d) --++(-1,+0.3);

			\draw [very thick](a) to[in=250, out=110] (c);
			\draw [very thick](a) to[in=290, out=70] (c);
			\draw [very thick](b) to[in=250, out=110] (d);
			\draw [very thick](b) to[in=290, out=70] (d);
			
		\end{tikzpicture}
	\end{center}
	where the graph is of twisted or product type and the value of $*$ is as discussed in Remark \ref{rem: labels explicit} (however both of the last data are irrelevant for the proof). Transport along the upper horizontal edge with label $(0,1,0,0)$ preserves left and right fiber edges in case \textbf{A} and swaps the edges in case \textbf{D} while introducing a sign to the transported weights $\mathrm{mod}\, (0,1,0,0)$ according to the signs in the case distinction. E.g.\ in the case \textbf{D}$+-$ we obtain
	\[(a,b,k,l)\equiv +(*,d,m,n)\mod (0,1,0,0),\qquad (c,d,m,n)\equiv - (*,b,k,l)\mod (0,1,0,0)\]
	 Thus transport along the $(0,1,0,0)$ edges forces the relations
	\begin{center}
	\begin{tabular}{c|c}
	case & relations \\ \hline
	\textbf{PA}$++$ & \\
	\textbf{PD}$++$ & $(k,l)=(m,n)$ \\
	\textbf{PA}$--$ & $(k,l)=-(k,l)$, $(m,n)=-(m,n)$ \\
	\textbf{TA}$+-$ & $(m,n)=-(m,n)$\\
	\textbf{TD}$+-$ & $(k,l)=(m,n)$, $(m,n)=-(k,l)$
	\end{tabular}
	\end{center}
	Under these relations the maximal possible dimension of the span of the weights at a fixed point is the dimension bound claimed in the statement of the theorem. Indeed the former leads to the latter since the dimension of the span of the weights (at one vertex) agrees with the maximal dimension of an effective action of a subtorus. 
	
It remains to argue that such an extension indeed exists on the $M_{a,b,c,d}^*$ in the respective cases.	
The case \textbf{PA}$++$ is just the product case which evidently has an extension to a $T^4$-action. In case \textbf{PD}$++$ (see Section \ref{subsec:PD++}) only the left hand factor of $\SO(4)\cong \SU(2)\times\SU(2)/\ZZ_2$ is used to act on the fiber $S^4$ during the clutching construction. Hence the maximal circle of the right hand factor acts on all fibers. This commutes with the given $T^2$-action and hence extends this to a $T^3$-action. It is indeed effective since the additional circle acts trivially on the base and effectively on every fiber, while the original $T^2$ acted effectively on the base. In the case \textbf{TA}$--$ the clutching happens via transformations in $\SO(3)\subset \SO(5)$ embedded in the lower right corner. Hence in this case the upper left $\SO(2)$ provides an additional circle acting on all fibers and extending the previous $T^2$-action. The extension is effective for the same reasoning as above.
\end{proof}
\begin{rem} A GKM action is called GKM$_k$ if any set of $k$ adjacent weights is linearly independent. Choosing $k,l,m,n$ such that the maximal dimension is reached we observe that this maximal extension is GKM$_4$ in case \textbf{PA}$++$, GKM$_3$ in case \textbf{PD}$++$, and GKM$_2$ in the remaining cases.
\end{rem}
\begin{rem} The general question whether the GKM graph of a given GKM action admits an effective extension of the labeling to a larger torus was addressed in \cite{KurokiUpperBounds}, for signed GKM graphs via the so-called group of axial functions. In our situation the graph is simple enough for an ad hoc solution of this problem.
\end{rem}

\bibliography{GKMPetrie}
\bibliographystyle{plain}
\end{document}